\documentclass[12pt,a4paper]{amsart}
\usepackage[utf8]{inputenc}
\usepackage[T1]{fontenc}

\usepackage{fancyhdr}
\usepackage{lipsum}

\newtheorem{theorem}{Theorem}[section]
\newtheorem{corollary}[theorem]{Corollary}
\newtheorem{lemma}[theorem]{Lemma}
\newtheorem{proposition}[theorem]{Proposition}

\newtheorem{conjecture}[theorem]{Conjecture}
\theoremstyle{definition}
\newtheorem{definition}[theorem]{Definition}
\newtheorem{example}[theorem]{Example}

\usepackage{amsfonts}
\usepackage[figuresright]{rotating}
\usepackage{amssymb}
\usepackage{amsmath}
\usepackage{graphics}
\usepackage{enumerate}
\usepackage{tikz}
\usepackage{caption}
\usepackage{subcaption}
\usepackage{float}
\usepackage{comment}

\DeclareMathOperator{\QEC}{QEC}

\title[On  quadratic embeddability of bipartite graphs]{On  quadratic embeddability of  bipartite graphs and  theta graphs}

\begin{document}
\today

\author{Wojciech Młotkowski, Marek Skrzypczyk and Michał Wojtylak}

\begin{abstract}
We compute the quadratic embedding constant for  complete bipartite graphs with disjoint edges removed.
Moreover, we study the quadratic embedding property for theta graphs, i.e., graphs consisting of three paths with common initial points and common endpoints.
As a result, we provide an infinite family of primary graphs which are not quadratically embeddable.
\end{abstract}
\maketitle

\section*{Introduction}
\subsection{The problem}

We consider undirected graphs $G=(V,E)$ without self-loops, 
with the standard distance $d(x,y)$ defined on the set $V\times V$ (we refer the reader to the Preliminaries for formal statements of the basic definitions). 
A graph $H$ is called isometrically embeddable in $G$ if $H$ is (isomorphic with) a subgraph of $G$ and the distances between vertices in $H$ are preserved when considered in $G$. 
A special attention will lie in bipartite graphs. Besides the complete bipartite graph $K_{m,n}$ we recall another important example of a bipartite graph:  the \emph{hypercube}. Namely, let $S$ be a finite set and let $H(S)$ be a graph such that the vertices are the subsets  of $S$; further, two vertices $N_1$ and $N_2$ are adjacent if and only if the symmetric difference $N_1 \triangle N_2$ is a singleton set.
In a seminal  paper \cite{Djokovic} Djokovi\'c characterised graphs isometrically embeddedable in a hypercube by a condition that $G$ is bipartite and the set
$$
G(a,b)=\{x\in G\  |\  d(x,a)<d(x,b)\},
$$
where vertices $a$ and $b$ are adjacent, is convex for all $a,b\in G$ (i.e., if for all $x,y\in G(a,b)$, $d(x,z)+d(y,z)=d(x,y)$ implies that $z\in G(a,b)$). The work has found numerous applications in graph theory and other branches of mathematics, see e.g., \cite{Asr1998,Deza1997,Ham2011}, and chemistry, e.g., \cite{Aro2019,Klav2008,  Klav1997}.

Later on Roth and Winkler \cite{RothWinkler} showed that a bipartite graph $G=(V,E)$ is  isometrically embeddable in a hypercube if and only if it admits a \emph{quadratic embedding}, i.e.,  a map $\varphi: V \to \mathcal{H}$ satisfying
$$
\|\varphi(x) - \varphi(y)\|^2=d(x,y), \quad x,y \in V, 
$$
where $(\mathcal H,\|\cdot\|)$ is some Hilbert space  and $d(x,y)$ is the graph distance.

This observation constitutes a  connection with the general theory of quadratic embeddings by 
Schoenberg, see \cite{Sch2}. In particular, a finite set $x_1,\dots,x_n$ with an arbitrary metric $d$  is quadratically embeddable if and only if the matrix
$$
q^D: = [q^{d(x_i,x_j)}]_{ij=1}^n ,\quad \text{where } D=[d(x_i,x_j]_{ij=1}^n,\ 0^0:=1, 
$$
is positive semidefinite for $q\in[0,1]$. This question is in connection with another classical task in harmonic analysis, which is to determine for a given symmetric matrix $D$ (with zeros on the diagonal only) 
 the following set
$$
\pi(G)= \{q \in \mathbb{R}: q^D \succeq 0\},
$$
where $\succeq 0$ means positive semidefinite. The set $\pi(G)$ is nonempty since $0, 1 \in \pi(G)$ and is closed and if $n>1$ we have $\pi(D)\subseteq[-1,1]$ due to existence of a principal minor of $q^D$ of the form
$\begin{bmatrix}
1 & q^x\\
q^x & 1 
\end{bmatrix}$ $(x\neq0)$.  The concept of quadratic embedding (of a metric space in general) has been studied  considerably
along with Euclidean distance geometry \cite{Alfakih,Bal,Bapat,Winkler,Jak1,Jak2,Maehara}   
and also appeared in harmonic analysis, for example in studying the Kazhdan property $(T)$ of groups \cite{Boz2, BozDolEjsGal, BozJanSpa, Obata5, Obata6}, and related structures  and  quantum probability \cite{Boz2, BozJanSpa, Hora1, Obata5, Obata6}.

Returning to the topic of graphs, the theory of quadratic embeddings for graphs was undertaken by Obata. In a series of coauthored papers \cite{ObataBaskoro2024,Mlo2,MłotkowskiObata2024,Obata6,Obata4,Obata1,ObataPrim6v,Obata3} he constituted a general theory for determining whether a graph (not necessarily bipartite) admits a quadratic embedding. Several of these results will be cited and used later on in the current paper.
A graph $G$ is said to be \emph{of QE class}  (\emph{of non-QE class}) if it admits 
a quadratic embedding  (does not admit a quadratic embedding, respectively), see \cite{Obata3}. Moreover, a graph of non-QE class is called \emph{primary} if it contains no isometrically embedded proper subgraphs of non-QE class \cite{Obata1}.\\ 

{\bf Main problem.} Obata showed all primary non-QE graphs on 5 and 6 vertices, as well as all complete multipartite primary non-QE graphs \cite{Obata1, ObataPrim6v, Obata3}. The main problem is to find all primary non-QE graphs and whether they exist for each number of vertices greater than or equal 5.\\

We provide a positive answer to this question.
Our main tool for achieving this will be the following 
improvement of the Djokovi\'c result. Namely,  it was asked by Bo\.zejko and answered in Tanaka \cite{Tan}, to characterise the case $[-1,1]\subseteq \pi(D)$ for a graph distance matrix $D$.  We present the result summarising the research discussed above. 
\begin{theorem}\label{tan-cite}
Let $G$ be a connected  graph with at least two vertices, then the following are equivalent:
\begin{enumerate}[\rm (i)]
\item $G$ is isometrically embeddable in a hypercube,
\item $G$ is bipartite and $G(a,b)$ is convex for each adjacent vertices $a,b\in G$,
\item $\pi(G)=[-1,1]$,
\item $G$ is bipartite and has no quintuple of vertices $(v_1,\dots,v_5)$ with $\{v_1,v_2\},\{v_3,v_4\}\in E$,
$d(v_1,v_3)=d(v_2,v_4)=d(v_1,v_4)-1=d(v_2,v_3)-1$, $d(v_5,v_2)=d(v_5,v_1)+1$, $d(v_5,v_3)=d(v_5,v_4)+1$.
\end{enumerate}
Furthermore, if a quintuple as in {\rm(iv)} exists, then $G$ is of non-QE class, regardless if it is bipartite or not. 
\end{theorem}
 
We will extend the result by providing a concept of another quintuple of vertices, working for non-bipartite graphs.   

Another important notion in the theory is the  \emph{QE constant} of a graph $G$ is defined by Obata in \cite{Obata3} as 
	$$
	\QEC(G)= \max \{  f^\top Df :  f^\top f =1, \textbf{1}^\top f =0 \},
	$$
where $\textbf{1}$ stands for the vector with all coefficients equal to 1.  
Apparently, a graph $G$ is of QE-class if and only if   
$\QEC(G) \le 0$ (in other words the distance matrix is \emph{conditionally negative}).  Properties of $\QEC(G)$ and its values for some types of graphs has been studied by several authors \cite{Bas, ObataBaskoro2024, ChouNan, IraSug, Jak1, Jak2, Lou, Mlo1, Mlo2, MłotkowskiObata2024, Obata1, ObataPrim6v, Obata3, Pur}.
In the current work we will provide or estimate the QE-constant for several graphs that have not been studied so far.  

\subsection{Our aproach and results} The intuition behind the research presented in the paper is the following. Recall that the complete bipartite graph $K_{m,n}$ is QE class only for $n=1$, or $m\ge 1$, or
		 $m=n=2$ and in general its $\QEC$ constant equals $\frac{2(m-1)(n-1)-2}{m+n}$,  see Obata--Zakiyyah \cite{Obata3}. 
	 Suppose now we take large $m,n$ and  remove consecutively its edges in a way that leaves the remaining graph connected. We receive in this way a family of graphs with the order given by inclusion (which, in general, is not the isometric embedding) with the maximal element $K_{m,n}$ and several minimal elements. Among these minimal elements we may find graphs of QE class, e.g., the path  $P_n$ with $QEC(P_n)=-\frac{1}{1+\cos{\pi/n}}$ (cf. \cite{Mlo1}). Our (successful) guess was that somewhere in between  lies hidden  the primary non-QE graph.  The  results of the research are as follows.
\begin{enumerate}
    \item In Section~\ref{sec:qeobg} we compute the QE constant for almost complete bipartite graph $K_{m,n}^t$, i.e. the complete bipartite graph $K_{m,n}$, with $t$ disjoint edges removed
 (Theorems \ref{th:Ktmn}, \ref{th:K22n}).

\item In Sections \ref{sec:tgwt}–\ref{se:proofs} we study the QE property for theta graphs $\Theta(\alpha,\beta,\gamma)$, see Definition \ref{Deftheta}. An important step was characterization of those theta graphs
 which admit Tanaka quintuple (or modified Tanaka quintuple), see Theorem \ref{th:thetawith5t}
 (and Theorem \ref{thetawithmodifiedtanaka}). The main results in this part are Theorems \ref{thm:thetaqe} and \ref{thm:thetanonqe}, which
 provide a sufficient and a necessary condition for $\Theta(\alpha,\beta,\gamma)$ to be of QE class.

\item  Finally, we prove that the family of primary non-QE graphs is infinite, namely,
 each theta graph which is not of QE type is primary, Theorem \ref{thm:primarynonqe} and Corollary~\ref{cor:primaryinfinite}.

\end{enumerate}
We conclude the Introduction with  tables of all subgraphs of  $K_{3,3}$, marking the instances where our results apply, see Table~\ref{tab:33}.

\begin{center}
\begin{table}[h]
\captionsetup{justification=centering}
\small
    \begin{tabular}{|c|c|c|c|c|c|}
    \hline
    
        No. & $G$ & $|E|$ & Type of graph & Obtained from & $\QEC (G)$ \\
        \hline
        73 & \begin{tikzpicture}[x=0.85cm,y=0.85cm] 
  \tikzset{     
    e4c node/.style={circle,fill,minimum size=0.10cm,inner sep=-1pt}, 
    e4c edge/.style={sloped,above,font=\footnotesize}
  }
  \node[e4c node] (1) at (0.00, 1.16) {}; 
  \node[e4c node] (2) at (0.50, 1.16) {}; 
  \node[e4c node] (3) at (1.00, 1.16) {}; 
  \node[e4c node] (4) at (0.00, 0.50) {}; 
  \node[e4c node] (5) at (0.50, 0.50) {}; 
\node[e4c node] (6) at (1.00, 0.50) {}; 

  \path[draw,thick]
  (1) edge[e4c edge]  (4)
  (1) edge[e4c edge]  (5)
  (2) edge[e4c edge]  (4)
  (2) edge[e4c edge]  (5)
  (3) edge[e4c edge]  (4)
  (3) edge[e4c edge]  (5)
  (1) edge[e4c edge]  (6)
  (2) edge[e4c edge]  (6)
  (3) edge[e4c edge]  (6)
  ;
\end{tikzpicture} & 9 & complete bipartite $K_{3,3}$ & & $1$ \\
        \hline
        55 & \begin{tikzpicture}[x=0.85cm,y=0.85cm] 
  \tikzset{     
    e4c node/.style={circle,fill,minimum size=0.10cm,inner sep=-1pt}, 
    e4c edge/.style={sloped,above,font=\footnotesize}
  }
  \node[e4c node] (1) at (0.00, 1.16) {}; 
  \node[e4c node] (2) at (0.50, 1.16) {}; 
  \node[e4c node] (3) at (1.00, 1.16) {}; 
  \node[e4c node] (4) at (0.00, 0.50) {}; 
  \node[e4c node] (5) at (0.50, 0.50) {}; 
\node[e4c node] (6) at (1.00, 0.50) {}; 

  \path[draw,thick]
  (1) edge[e4c edge]  (4)
  (1) edge[e4c edge]  (5)
  (2) edge[e4c edge]  (4)
  (2) edge[e4c edge]  (5)
  (3) edge[e4c edge]  (4)
  (3) edge[e4c edge]  (5)
  (1) edge[e4c edge]  (6)
  (2) edge[e4c edge]  (6)
  ;
\end{tikzpicture} & 8 & {$K_{3,3}^1$, cf. Cor. \ref{cor:ktmm} }& 73 & $\frac{-3+\sqrt{17}}{2}$ \\
        \hline
        36 & \begin{tikzpicture}[x=0.85cm,y=0.85cm] 
  \tikzset{     
    e4c node/.style={circle,fill,minimum size=0.10cm,inner sep=-1pt}, 
    e4c edge/.style={sloped,above,font=\footnotesize}
  }
  \node[e4c node] (1) at (0.00, 1.16) {}; 
  \node[e4c node] (2) at (0.50, 1.16) {}; 
  \node[e4c node] (3) at (1.00, 1.16) {}; 
  \node[e4c node] (4) at (0.00, 0.50) {}; 
  \node[e4c node] (5) at (0.50, 0.50) {}; 
\node[e4c node] (6) at (1.00, 0.50) {}; 

  \path[draw,thick]
  (1) edge[e4c edge]  (4)
  (1) edge[e4c edge]  (5)
  (2) edge[e4c edge]  (4)
  (2) edge[e4c edge]  (5)
  (3) edge[e4c edge]  (4)
  (1) edge[e4c edge]  (6)
  (2) edge[e4c edge]  (6)
  ;
\end{tikzpicture} & 7 & & 55 & $\thickapprox 0.408$ \\
        \hline
        35 & \begin{tikzpicture}[x=0.85cm,y=0.85cm]
  \tikzset{     
    e4c node/.style={circle,fill,minimum size=0.10cm,inner sep=-1pt}, 
    e4c edge/.style={sloped,above,font=\footnotesize}
  }
  \node[e4c node] (1) at (0.00, 1.16) {}; 
  \node[e4c node] (2) at (0.50, 1.16) {}; 
  \node[e4c node] (3) at (1.00, 1.16) {}; 
  \node[e4c node] (4) at (0.00, 0.50) {}; 
  \node[e4c node] (5) at (0.50, 0.50) {}; 
\node[e4c node] (6) at (1.00, 0.50) {}; 

  \path[draw,thick]
  (1) edge[e4c edge]  (4)
  (1) edge[e4c edge]  (5)
  (2) edge[e4c edge]  (4)
  (3) edge[e4c edge]  (4)
  (3) edge[e4c edge]  (5)
  (1) edge[e4c edge]  (6)
  (2) edge[e4c edge]  (6)
  ;
\end{tikzpicture} & 7 &  \begin{tabular}{c}
      $K_{3,3}^2$,  cf. Cor.\ref{cor:ktmm},\\
     also $\Theta(1,3,3)$, Prop.~3.15 in \cite{ObataPrim6v}. 
\end{tabular} &
 55 & 0 \\
        \hline
        19 & \begin{tikzpicture}[x=0.85cm,y=0.85cm] 
  \tikzset{     
    e4c node/.style={circle,fill,minimum size=0.10cm,inner sep=-1pt}, 
    e4c edge/.style={sloped,above,font=\footnotesize}
  }
  \node[e4c node] (1) at (0.00, 1.16) {}; 
  \node[e4c node] (2) at (0.50, 1.16) {}; 
  \node[e4c node] (3) at (1.00, 1.16) {}; 
  \node[e4c node] (4) at (0.00, 0.50) {}; 
  \node[e4c node] (5) at (0.50, 0.50) {}; 
\node[e4c node] (6) at (1.00, 0.50) {}; 

  \path[draw,thick]
  (1) edge[e4c edge]  (5)
  (2) edge[e4c edge]  (4)
  (3) edge[e4c edge]  (4)
  (3) edge[e4c edge]  (5)
  (1) edge[e4c edge]  (6)
  (2) edge[e4c edge]  (6)
  ;
\end{tikzpicture} & 6 &  
\begin{tabular}{c}crown graph, \\
also: cycle $C_6$, cf., \cite{Obata3}
\end{tabular}& 35 & 0 \\
        \hline
        
        18 & \begin{tikzpicture}[x=0.85cm,y=0.85cm]
  \tikzset{     
    e4c node/.style={circle,fill,minimum size=0.10cm,inner sep=-1pt}, 
    e4c edge/.style={sloped,above,font=\footnotesize}
  }
  \node[e4c node] (1) at (0.00, 1.16) {}; 
  \node[e4c node] (2) at (0.50, 1.16) {}; 
  \node[e4c node] (3) at (1.00, 1.16) {}; 
  \node[e4c node] (4) at (0.00, 0.50) {}; 
  \node[e4c node] (5) at (0.50, 0.50) {}; 
\node[e4c node] (6) at (1.00, 0.50) {}; 

  \path[draw,thick]
  (1) edge[e4c edge]  (4)
  (1) edge[e4c edge]  (5)
  (2) edge[e4c edge]  (4)
  (3) edge[e4c edge]  (4)
  (1) edge[e4c edge]  (6)
  (2) edge[e4c edge]  (6)
  ;
\end{tikzpicture} & 6 &  &
 36 or 35& 0 \\
        \hline
        
15 & \begin{tikzpicture}[x=0.85cm,y=0.85cm] 
  \tikzset{     
    e4c node/.style={circle,fill,minimum size=0.10cm,inner sep=-1pt}, 
    e4c edge/.style={sloped,above,font=\footnotesize}
  }
  \node[e4c node] (1) at (0.00, 1.16) {}; 
  \node[e4c node] (2) at (0.50, 1.16) {}; 
  \node[e4c node] (3) at (1.00, 1.16) {}; 
  \node[e4c node] (4) at (0.00, 0.50) {}; 
  \node[e4c node] (5) at (0.50, 0.50) {}; 
\node[e4c node] (6) at (1.00, 0.50) {}; 

  \path[draw,thick]
  (1) edge[e4c edge]  (4)
  (1) edge[e4c edge]  (5)
  (2) edge[e4c edge]  (4)
  (3) edge[e4c edge]  (5)
  (1) edge[e4c edge]  (6)
  (2) edge[e4c edge]  (6)
  ;
\end{tikzpicture} & 6 & & 
   36 or 35& 0 \\
        \hline    
        
6 & \begin{tikzpicture}[x=0.85cm,y=0.85cm] 
  \tikzset{     
    e4c node/.style={circle,fill,minimum size=0.10cm,inner sep=-1pt}, 
    e4c edge/.style={sloped,above,font=\footnotesize}
  }
  \node[e4c node] (1) at (0.00, 1.16) {}; 
  \node[e4c node] (2) at (0.50, 1.16) {}; 
  \node[e4c node] (3) at (1.00, 1.16) {}; 
  \node[e4c node] (4) at (0.00, 0.50) {}; 
  \node[e4c node] (5) at (0.50, 0.50) {}; 
\node[e4c node] (6) at (1.00, 0.50) {}; 

  \path[draw,thick]
  (1) edge[e4c edge]  (5)
  (2) edge[e4c edge]  (4)
  (3) edge[e4c edge]  (5)
  (1) edge[e4c edge]  (6)
  (2) edge[e4c edge]  (6)
  ;
\end{tikzpicture} & 5 &  
path $P_6$, cf. Thm.1.1 in \cite{Mlo1} &   19, 18 or 15& $2\sqrt{3}-4$ \\
        \hline          
        
    5 & \begin{tikzpicture}[x=0.85cm,y=0.85cm] 
  \tikzset{     
    e4c node/.style={circle,fill,minimum size=0.10cm,inner sep=-1pt}, 
    e4c edge/.style={sloped,above,font=\footnotesize}
  }
  \node[e4c node] (1) at (0.00, 1.16) {}; 
  \node[e4c node] (2) at (0.50, 1.16) {}; 
  \node[e4c node] (3) at (1.00, 1.16) {}; 
  \node[e4c node] (4) at (0.00, 0.50) {}; 
  \node[e4c node] (5) at (0.50, 0.50) {}; 
\node[e4c node] (6) at (1.00, 0.50) {}; 

  \path[draw,thick]
  (1) edge[e4c edge]  (4)
  (1) edge[e4c edge]  (5)
  (3) edge[e4c edge]  (4)
  (1) edge[e4c edge]  (6)
  (2) edge[e4c edge]  (6)
  ;
\end{tikzpicture} & 5 & & 
   18 or 15& $\thickapprox -0.4648$ \\
        \hline    
        
        3 & \begin{tikzpicture}[x=0.85cm,y=0.85cm] 
  \tikzset{     
    e4c node/.style={circle,fill,minimum size=0.10cm,inner sep=-1pt}, 
    e4c edge/.style={sloped,above,font=\footnotesize}
  }
  \node[e4c node] (1) at (0.00, 1.16) {}; 
  \node[e4c node] (2) at (0.50, 1.16) {}; 
  \node[e4c node] (3) at (1.00, 1.16) {}; 
  \node[e4c node] (4) at (0.00, 0.50) {}; 
  \node[e4c node] (5) at (0.50, 0.50) {}; 
\node[e4c node] (6) at (1.00, 0.50) {}; 

  \path[draw,thick]
  (1) edge[e4c edge]  (4)
  (1) edge[e4c edge]  (5)
  (2) edge[e4c edge]  (4)
  (3) edge[e4c edge]  (4)
  (1) edge[e4c edge]  (6)
  ;
\end{tikzpicture} & 5 & & 
   18 or 15& $\thickapprox -0.4385$ \\
        \hline   
    \end{tabular}
    \caption{Subgraphs of $K_{3,3}$  obtained by consecutive removing of edges. }
    \label{tab:33}
    \end{table}
\end{center}

\section{Preliminaries}
A \textit{graph} $G=(V,E)$ is a pair of non-empty set $V$ of vertices and a set $E$, where an edge is a set of two distinct vertices in $V$. A \textit{walk} is a finite  sequence of edges which joins a sequence of vertices. A graph is called \textit{connected} if there is a finite walk for each pair of vertices. 
In such case by $d(x,y)=d_G(x,y)$ we denote the length of a shortest walk connecting $x$ and $y$, such a walk is called \textit{geodesic}. Note that  $d(x,y)$ is a metric on $V$, called as the \textit{graph distance}. The \textit{distance matrix} is defined by
$$
D=d(x,y)_{x,y \in V}.
$$
We will identify $D$ with an $n\times n$ real matrix by fixing an ordering of vertices. By $\mathbf 1$ we understand the vector in $\mathbb R^n$ consisting of ones, the size $n$ will be clear from the context.  
A \textit{bipartite graph} is a graph whose vertex set can be partitioned into two subsets $V_1$ and $V_2$, such that no two vertices within the same set are adjacent.
The \textit{complete bipartite graph} $K_{m,n}$ is the maximal bipartite graph.

A \textit{quadratic embedding} of a finite connected graph $G=(V,E)$ in a Hilbert space $\mathcal{H}$ is a map $\varphi: V \to \mathcal{H}$ satisfying
$$||\varphi(x) - \varphi(y)||^2=d(x,y), \quad x,y \in V, $$
where $||\cdot||$ is the norm of the Hilbert space $\mathcal{H}$, and $d(x,y)$ is the graph distance.
A graph $G$ is said to be \textit{of QE class}  (\textit{of non-QE class}) if it admits 
a quadratic embedding  (does not admit a quadratic embedding, respectively).
The definition of $\QEC(G)$ was given in the Introduction.

A graph $H=(V',E')$ is called a \textit{subgraph} of $G=(V,E)$ if $V' \subset V$ and $E' \subset E$. Let both graphs are connected and hence admit graph distances of their own. We say that $H$ is \textit{isometrically embedded} in $G$ if 
	$$d_H (x,y)= d_G (x,y) \quad x,y \in V'.$$
		Let $G=(V,E)$ and $H=(V',E')$ be graphs and assume that $H$ is isometrically embedded in $G$. Note the following simple fact: if $H$ is of non-QE class, so $G$ is of non-QE class as well, furthermore $\QEC(H)\le \QEC(G)$.
		Hence, non-QE graphs can be divided into two different types. A graph of non-QE class is called \textit{primary} if it contains no isometrically embedded proper subgraphs of non-QE class \cite{Obata1}. Otherwise, it is called \textit{non-primary}.

Let $G_1=(V_1, E_1)$ and $G_2=(V_2,E_2)$ have a distinguished vertices $v_1\in V_1$ and $v_2\in V_2$, respectively. A \textit{star product} denoted by $G_1 \star G_2$ is a graph obtained by gluing vertices $v_1$ and $v_2$. Recall that if  both $G_1$ and $G_2$ are of QE class, then $G_1\star G_2$ is also of QE class, cf. \cite{Obata3}.

We end the section by remarking that the $\QEC$ constant is not (weakly) monotone with respect to removing disjoint edges from a bipartite graph.  
\begin{example}

Consider Table \ref{tab1}. (The graph on the right hand side, denoted  later on in this paper as $\Theta(2,2,4)$, is particularly interesting, because it is of primary non-QE class.)

\begin{center}
\begin{table}[h]
\captionsetup{justification=centering}
\small
    \begin{tabular}{|c||c|c|c|}
    \hline  & 
\begin{tikzpicture}[x=4cm,y=4cm] 
  \tikzset{     
    e4c node/.style={circle,draw,minimum size=0.45cm,inner sep=0}, 
    e4c edge/.style={sloped,above,font=\footnotesize}
  }
  \node[e4c node] (u1) at (0.00, 1.00) {$u_1$}; 
  \node[e4c node] (u2) at (0.30, 1.00) {$u_2$}; 
  \node[e4c node] (u3) at (0.60, 1.00) {$u_3$}; 
  \node[e4c node] (u4) at (0.90, 1.00) {$u_4$}; 
  \node[e4c node] (u5) at (0.15, 0.50) {$u_5$}; 
  \node[e4c node] (u6) at (0.45, 0.50) {$u_6$}; 
  \node[e4c node] (u7) at (0.75, 0.50) {$u_7$};

  \path[draw,thick]
  (u1) edge[e4c edge]  (u5)
  (u1) edge[e4c edge]  (u6)
  (u1) edge[e4c edge]  (u7)
  (u2) edge[e4c edge]  (u6)
  (u2) edge[e4c edge]  (u7)
  (u3) edge[e4c edge]  (u5)
  (u3) edge[e4c edge]  (u7)
  (u4) edge[e4c edge]  (u5)
  (u4) edge[e4c edge]  (u6)
  ;
\end{tikzpicture}
& 
\begin{tikzpicture}[x=4cm,y=4cm] 
  \tikzset{     
    e4c node/.style={circle,draw,minimum size=0.45cm,inner sep=0}, 
    e4c edge/.style={sloped,above,font=\footnotesize}
  }
  \node[e4c node] (u1) at (0.00, 1.00) {$v_1$}; 
  \node[e4c node] (u2) at (0.30, 1.00) {$v_2$}; 
  \node[e4c node] (u3) at (0.60, 1.00) {$v_3$}; 
  \node[e4c node] (u4) at (0.90, 1.00) {$v_4$}; 
  \node[e4c node] (u5) at (0.15, 0.50) {$v_5$}; 
  \node[e4c node] (u6) at (0.45, 0.50) {$v_6$}; 
  \node[e4c node] (u7) at (0.75, 0.50) {$v_7$};

  \path[draw,thick]
  (u1) edge[e4c edge]  (u5)
  (u1) edge[e4c edge]  (u6)
  (u1) edge[e4c edge]  (u7)
  (u2) edge[e4c edge]  (u5)
  (u2) edge[e4c edge]  (u7)
  (u3) edge[e4c edge]  (u5)
  (u3) edge[e4c edge]  (u6)
  (u4) edge[e4c edge]  (u5)
  (u4) edge[e4c edge]  (u6)
  ;
\end{tikzpicture}

& 
\begin{tikzpicture}[x=4cm,y=4cm] 
  \tikzset{     
    e4c node/.style={circle,draw,minimum size=0.35cm,inner sep=0}, 
    e4c edge/.style={sloped,above,font=\footnotesize}
  }
  \node[e4c node] (u1) at (0.00, 1.00) {$w_1$}; 
  \node[e4c node] (u2) at (0.30, 1.00) {$w_2$}; 
  \node[e4c node] (u3) at (0.60, 1.00) {$w_3$}; 
  \node[e4c node] (u4) at (0.90, 1.00) {$w_4$}; 
  \node[e4c node] (u5) at (0.15, 0.50) {$w_5$}; 
  \node[e4c node] (u6) at (0.45, 0.50) {$w_6$}; 
  \node[e4c node] (u7) at (0.75, 0.50) {$w_7$};

  \path[draw,thick]
  (u1) edge[e4c edge]  (u6)
  (u1) edge[e4c edge]  (u7)
  (u2) edge[e4c edge]  (u5)
  (u2) edge[e4c edge]  (u7)
  (u3) edge[e4c edge]  (u5)
  (u3) edge[e4c edge]  (u6)
  (u4) edge[e4c edge]  (u5)
  (u4) edge[e4c edge]  (u6)
  ;
\end{tikzpicture}

\\\hline
    $\QEC$ & 0 & $\approx 0.5149$ & $\approx 0.5529$ \\ \hline
    \end{tabular}
     \caption{The graph on the right can be obtained by removing the edge $u_1 u_5$ of the graph on the left or the edge $v_1 v_5$ of the graph in the middle  }
    \label{tab1}
    \end{table}
    \end{center}
\end{example}

%%%%%%%%%%%%%%%%%%%%%%%%%%%%%%%%%%%%%%%%%%%%%%%%%%%%%%%%%%%%%%%%%%%%%%%%%%%%%%%%%%%%%%%%%%%%%%%%%%%%%%%%%%%%%%%%%%%%%%%%%%%%%%%%%%

\section{$QE$ constants of almost complete bipartite graphs} \label{sec:qeobg}

%%%%%%%%%%%%%%%%%%%%%%%%%%%%%%%%%%%%%%%%%%%%%%%%%%%%%%%%%%%%%%%%%%%%%%%%%%%%%%%%%%%%%%%%%%%%%%%%%%%%%%%%%%%%%%%%%%%%%%%%%%%%%%%%%%%%%%%%%%%%%%

For $0\le t\le m\le n$, $m\ge1$, define
$K_{m,n}^{t}:=(V,E)$, where $V:=\{u_1,\ldots,u_m,v_1,\ldots,v_n\}$,
\[
E:=\big\{\{u_i,v_j\}:1\le i\le m,1\le j\le n, \hbox{ and if } i=j\hbox{ then } i>t\big\}.
\]
The graphs $K_{m,n}^t$ are called \textit{almost complete bipartite graphs}.

In particular, for $t=0$ we denote $K_{m,n}$ and this graph is the \textit{complete bipartite graph} with parameters $m,n$
and $K_{m,m}^{m}$ is called the \textit{crown graph} of order $m$.
The graphs $K_{1,n}^{1}$, $K_{2,2}^{2}$ are not connected, $K_{1,1}=P_2$, $K_{1,2}=P_3$, $K_{2,2}=C_{4}$, $K_{2,2}^{1}=P_{4}$, $K_{2,3}^{2}=P_5$, $K_{3,3}^{3}=C_6$ and $K_{4,4}^{4}$ is the skeleton of the cube.
Observe that the graphs $K_{m,m}$, $m\ge1$, and $K_{m,m}^{m}$, $m\ge3$, are \textit{transmission regular},
i.e. the vectors $\mathbf 1$ and $D\mathbf 1$ are linearly dependent.

\begin{figure}[h]\label{figure:e4c}
\centering
    \begin{tikzpicture}[x=3cm,y=3cm] 
  \tikzset{
    e4c node/.style={circle,draw,minimum size=0.3cm,inner sep=0},
    e4c edge/.style={sloped,above,font=\footnotesize}
  }
  \node[e4c node] (1) at (0.00, 1.60) {};
  \node[e4c node] (2) at (0.00, 1.20) {};
  \node[e4c node] (3) at (0.00, 0.80) {};
  \node[e4c node] (4) at (0.00, 0.40) {};
  \node[e4c node] (5) at (0.00, 0.00) {};
  \node[e4c node] (10) at (1.00, 0.00) {};
  \node[e4c node] (9) at (1.00, 0.40) {};
  \node[e4c node] (8) at (1.00, 0.80) {};
  \node[e4c node] (7) at (1.00, 1.20) {};
  \node[e4c node] (6) at (1.00, 1.60) {};

  \path[draw,thick]
  (2) edge[e4c edge]  (8)
  (2) edge[e4c edge]  (6)
  (1) edge[e4c edge]  (10)
  (1) edge[e4c edge]  (9)
  (1) edge[e4c edge]  (8)
  (1) edge[e4c edge]  (7)
  (2) edge[e4c edge]  (9)
  (2) edge[e4c edge]  (10)
  (3) edge[e4c edge]  (6)
  (3) edge[e4c edge]  (7)
  (3) edge[e4c edge]  (9)
  (3) edge[e4c edge]  (10)
  (4) edge[e4c edge]  (6)
  (4) edge[e4c edge]  (7)
  (4) edge[e4c edge]  (8)
  (4) edge[e4c edge]  (10)
  (5) edge[e4c edge]  (6)
  (5) edge[e4c edge]  (7)
  (5) edge[e4c edge]  (8)
  (5) edge[e4c edge]  (9)
  ;
\end{tikzpicture}
\caption{The crown graph $K_{5,5}^{5}$}
\end{figure}
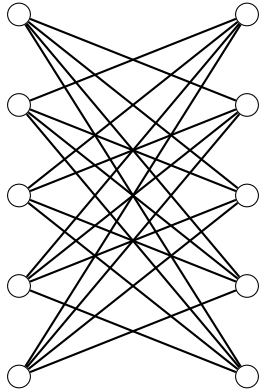

It is was proven in \cite{Obata3} that
\begin{equation}\label{eq:completebi}
\QEC(K_{m,n})=\frac{2mn-2m-2n}{mn}.
\end{equation}
We are going to study the QE constant of $K_{m,n}^{t}$ for $t\ge1$.
Note that in $K_{2,n}^{2}$, $n\ge3$, we have $d(v_1,v_2)=4$,
while $d(v_i,v_j)=2$ for $1\le i<j\le n$, in all the other connected graphs $K_{m,n}^{t}$.
Therefore we have to study two subclasses separately.
Also the case $(t,m,n)=(3,3,3)$ has to be distinguished, for which we know
that $\QEC(K_{3,3}^{3})=\QEC(C_6)=0$, see~\cite{Obata3}.
Accordingly, we are going to prove two theorems.

\begin{theorem}\label{th:Ktmn}
Assume that $1\le t\le m\le n$, $(t,m,n)\ne(3,3,3)$ and that if $t=1$ or $t=2$ then $t<m$. Then
\begin{equation}
\QEC(K_{m,n}^{t})=2\lambda_0-2,
\end{equation}
where $\lambda_0$ is the maximal root of the polynomial
\begin{equation}
P(\lambda):=(m+n)\lambda^3+(2t-mn)\lambda^2+(2t-m-n)\lambda+(m-t)(n-t).
\end{equation}
\end{theorem}

\begin{theorem}\label{th:K22n}
For $n\ge3$ we have
\begin{equation}\label{eq:thK22n}
\QEC(K_{2,n}^{2})=\frac{n-8+\sqrt{5n^2-24 n+32}}{n+2}.
\end{equation}
\end{theorem}

\begin{proof}[Proof of Theorem~\ref{th:Ktmn}.]
If $(t,m,n)$ satisfies the assumptions then the distance matrix of $K_{m,n}^t$ has the following form:
\[D=
\left(
\begin{array}{c|c}
   2\cdot\mathbf{ 1}_{m,m}-2\cdot I_m  & \begin{array}{cc}
    \mathbf{1}_{t,t}+2\cdot I_{t} &
    \mathbf{1}_{t,n-t} \\ \mathbf{1}_{m-t,t} & \mathbf{1}_{m-t,n-t}\end{array} \\ \hline
\begin{array}{cc}
    \mathbf{1}_{t,t}+2\cdot I_{t} &
    \mathbf{1}_{t,m-t} \\ \mathbf{1}_{n-t,t} & \mathbf{1}_{n-t,m-t}\end{array} & 2\cdot\mathbf{ 1}_{n,n}-2\cdot I_n
\end{array}
\right).
\]

Let $f=(x_1,\ldots,x_m, y_1,\ldots y_n)^\top $, with $f^\top f =1$ and $f^\top {\mathbf 1}=0$.
Then, since $\sum_{j=1}^{n}y_j=-\sum_{i=0}^m x_i$, we get
\begin{multline*}
f^\top Df= 2\Big( \big(\sum_{i=1}^m x_i\big)^2 - \sum_{i=1}^m x_i^2 \Big) + 2\sum_{i=1}^m x_i \sum_{j=1}^n y_j + 4 \sum_{i=1}^t x_i y_i + 2\Big(\big(\sum_{j=1}^n y_j\big)^2 - \sum_{j=1}^n y_j^2 \Big)\\
= -2 + 2 \Big (\big(\sum_{i=1}^m x_i\big)^2 + 2\sum_{i=1}^t x_i y_i\Big).
\end{multline*}

The usual method of Lagrange multipliers will be used to find the maximum above,  cf. Proposition 4.1 in \cite{Obata3}. Our Lagrangian function has the form
$$
\mathcal{L}(f, \lambda, \mu)= \Big(\sum_{k=1}^m x_k\Big)^2 + 2\sum_{i=1}^t x_i y_i - \lambda \left(f^\top f  -1\right) - \mu f^\top {\mathbf 1}.
$$
Comparing the partial derivatives to zero and putting $X:=\sum_{i=1}^m x_i$, we obtain:
\begin{align}
0=\frac{\partial \mathcal{L}}{\partial x_i} &= 2X +2y_i -2\lambda x_i -\mu\qquad\hbox{for $i\le t$,}\label{2xi}\\
0=\frac{\partial \mathcal{L}}{\partial x_i} &= 2X -2\lambda x_i -\mu \qquad\hbox{for $t<i\le m$,} \label{2xj}\\
0=\frac{\partial \mathcal{L}}{\partial y_i} &= 2x_i -2\lambda y_i -\mu\qquad\hbox{for $i\le t$,} \label{2yi}\\
0=\frac{\partial \mathcal{L}}{\partial y_j} &= -2\lambda y_j -\mu\qquad\hbox{for $t<j\le n$,} \label{2yj}\\
0=\frac{\partial \mathcal{L}}{\partial \lambda}&= x_1^2+\ldots +x_m^2+y_1^2+\ldots+y_n^2 -1,\label{2lambda}\\
0=\frac{\partial \mathcal{L}}{\partial \mu}&= x_1+\ldots +x_m+y_1+\ldots+y_n.\label{2mu}
\end{align}

Let $\mathcal{S}(D)$ be the set of all $(f,\lambda,\mu)\in\mathbb{R}^{m+n}\times\mathbb{R}\times\mathbb{R}$
satisfying (\ref{2xi}), (\ref{2xj}), (\ref{2yi}), (\ref{2yj}), (\ref{2lambda}) and (\ref{2mu})
and let $\mathcal{S}_{\lambda}(D):=\{\lambda:(f,\lambda,\mu)\in\mathcal{S}(D)\}$, $\lambda_{\max}:=\max\mathcal{S}_{\lambda}(D)$.
Then $\QEC(G)=-2+2\lambda_{\max}$.
For any graph which is not complete we have $\QEC(G) > -1$, cf.  \cite{Bas}, thus $\lambda_{\max} > \frac{1}{2}$.

\textit{Claim.} We have $1\in\mathcal{S}_{\lambda}(D)$ if and only if either $(t,m,n)=(1,2,3)$ or $t\ge2$.

If $t\ge2$ then we can put $\mu:=0$, $\lambda:=1$, $x_1=y_1:=1/2$, $x_2=y_2:=-1/2$,
$x_i=y_j:=0$ for $3\le i\le m$, $3\le j\le n$ and all equations \eqref{2xi}--\eqref{2mu} are satisfied.

Now assume that $t=1$.
Then $m\ge2$ and \eqref{2xi}, \eqref{2yi} imply that $X=\mu$, which, in turn, implies
that $x_i=\mu/2$ for $2\le i\le m$. This leads to $x_1=(3-m)\mu/2$.
Similarly, $y_j=-\mu/2$ for $2\le j\le n$ and $y_1=(n-3)\mu/2$ .
Now \eqref{2xi} yields that either $m+n=5$ or $\mu=0$, however the latter contradicts \eqref{2lambda}.
Therefore $m=2,n=3$.
In this case $\mu$ can be chosen so that \eqref{2lambda} is satisfied.
This concludes the proof of the claim.

Now assume that $(f,\lambda,\mu)\in\mathcal{S}(D)$ and $\lambda>1/2$, $\lambda\ne1$. %$\lambda\ne0,\pm1$.
Solving the system \eqref{2xi} and \eqref{2yi} for fixed $1\le i\le t$ we obtain
\begin{equation}\label{eq:ilet}
x_i=\frac{\lambda}{\lambda^2-1}X-\frac{\mu}{2\lambda-2},\qquad
y_i=\frac{1}{\lambda^2-1}X-\frac{\mu}{2\lambda-2},
\end{equation}
while \eqref{2xj}, \eqref{2yj} imply that for $t<i\le m$, $t<j\le n$
\begin{equation}\label{eq:iget}
x_i=\frac{1}{\lambda}X-\frac{\mu}{2\lambda},\qquad
y_j=\frac{-\mu}{2\lambda}.
\end{equation}
Now we recall that
\begin{equation}\label{eq:xxxxxy}
X=\sum_{i=1}^{m}x_i,\qquad -X=\sum_{j=1}^{n}y_j.
\end{equation}

Assume that $t=m=n\ge4$.
From \eqref{eq:ilet}, \eqref{eq:xxxxxy} we get system of equation:
\[
\frac{m\lambda X}{\lambda^2-1}-\frac{m\mu}{2\lambda-2}=X,\qquad
\frac{m X}{\lambda^2-1}-\frac{m\mu}{2\lambda-2}=-X,
\]
which has a nonzero solution $(X,\mu)$ if and only if
\[
2\lambda^2-m\lambda+m-2=0,
\]
hence $\lambda_{\max}=m/2-1$.
Note that in this case $P(\lambda)=m\lambda^2(2\lambda-m+2)$,
hence $m/2-1$ is indeed the largest root of $P(\lambda)$.

Now we assume that $2\le t\le m\le n$, $t<n$.
From \eqref{eq:ilet}, \eqref{eq:iget}, \eqref{eq:xxxxxy} we get equations
\[
\frac{\lambda^3-m\lambda^2-\lambda+m-t}{\lambda(\lambda^2-1)}X=\frac{-(m\lambda-m+t)\mu}{2\lambda(\lambda-1)}.
\]
\[
\frac{\lambda^2+t-1}{\lambda^2-1}X=\frac{(n\lambda-n+t)\mu}{2\lambda(\lambda-1)}.
\]
Note that $n\lambda-n+t>0$, $\lambda^2+t-1>0$, as $\lambda>1$, therefore $X=0$ if and only if $\mu=0$,
in this case however all $x_i,y_j$ would be $0$, which contradicts \eqref{2lambda}. Hence $X\ne0\ne\mu$.

From the second equation
\begin{equation}\label{eq:xxx}
X=\frac{(\lambda+1)(n\lambda-n+t)\mu}{2\lambda(\lambda^2+t-1)},
\end{equation}
hence for $1\le i\le t$
\begin{equation}\label{eq:xxyyii}
x_i=-\frac{(\lambda-n+1)\mu}{2\left(\lambda^2+t-1\right)},\qquad
y_i=-\frac{\left(\lambda^2+\lambda-n+t\right)\mu}{2\lambda \left(\lambda^2+t-1\right)}
\end{equation}
and for $t<i\le m$, $t<j\le n$
\begin{equation}\label{eq:xxyyij}
x_i=-\frac{\left(\lambda^3-n\lambda^2-\lambda+n-t\right)\mu}{2\lambda^2 \left(\lambda^2+t-1\right)},\qquad
y_j=\frac{-\mu}{2\lambda}.
\end{equation}
Equations \eqref{eq:xxx}, \eqref{eq:xxyyii}, \eqref{eq:xxyyij} imply
\eqref{2xi},  \eqref{2xj}, \eqref{2yi}, \eqref{2yj} and $\sum_{j=1}^{n} y_j=-X$.
In addition, the~equality $\sum_{i=1}^{m}x_i=X$, and consequently \eqref{2mu}, holds if and only if
\[
t\lambda^2(\lambda-n+1)+(m-t)\left(\lambda^3-n\lambda^2-\lambda+n-t\right)=-\lambda(\lambda+1)(n\lambda-n+t),
\]
which is equivalent to $P(\lambda)=0$.

Since $P(1)=t(t+4-m-n)$, we have $P(1)>0$ if and only if $(t,m,n)=(1,2,2)$.
In this case $1\notin\mathcal{S}_{\lambda}(D)$, $\lambda_{\max}=\sqrt{2}/2$
and $\QEC(K_{2,2}^1)=\QEC(P_4)=\sqrt{2}-2$, cf.~\cite{Mlo1}.
Moreover, $P(1)=0$ if and only if $(t,m,n)$
is either $(1,2,3)$, or $(2,3,3)$, or $(3,3,4)$, or $(4,4,4)$.
In all these cases the maximal root of $P(\lambda)$ is $1\in\mathcal{S}_{\lambda}(D)$,
and therefore $\lambda_{\max}=\lambda_0=1$. In all the other cases $P(1)<0$, and therefore
$\lambda_{\max}=\lambda_0>1$.
\end{proof}

\begin{proof}[Proof of Theorem~\ref{th:K22n}.]
Now consider $K_{2,n}^{2}$, $n\ge3$. For $f:=(x_1,x_2,y_1,\ldots,y_n)^\top\in\mathbb{R}^{2+n}$
we have
\begin{multline*}
f^\top Df=4x_1 x_2+6(x_1 y_1+x_2 y_2)+2(x_1 y_2+x_2 y_1)+8y_1 y_2
+2(x_1+x_2)(y_3+\ldots+y_n)\\
\phantom{aaaa}+4(y_1+y_2)(y_3+\ldots+y_n)+2(y_3+\ldots+y_3)^2-2(y_3^2+\ldots+y_n^2)\\
=4x_1 x_2+6(x_1 y_1+x_2 y_2)+2(x_1 y_2+x_2 y_1)+8y_1 y_2
-2(x_1+x_2)(x_1+x_2+y_1+y_2)\\
-4(y_1+y_2)(x_1+x_2+y_1+y_2)+2(x_1+x_2+y_1+y_2)^2+2(x_1^2+x_2^2+y_1^2+y_2^2)-2\\
=2x_1^2+2x_2^2+4x_1 x_2+4x_1 y_1+4x_2 y_2+4y_1 y_2-2.
\end{multline*}
Putting
\[
\mathcal{L}(f,\lambda,\mu):=x_1^2+x_2^2+2x_1 x_2+2x_1 y_1+2x_2 y_2+2y_1 y_2- \lambda (f^\top f  -1) - \mu f^\top {\mathbf 1}
\]
we have
\begin{align}
0=\frac{\partial \mathcal{L}}{\partial x_i} &= 2x_1+2x_2+2y_{i} -2\lambda x_i -\mu
\qquad\hbox{for $i=1,2$,}\label{22xi}\\
0=\frac{\partial \mathcal{L}}{\partial y_i} &= 2x_i +2y_{3-i}-2\lambda y_i -\mu
\qquad\hbox{for $i=1,2$,} \label{22xj}\\
%0=\frac{\partial \mathcal{L}}{\partial y_i} &= 2x_i -2\lambda y_i -\mu\qquad\hbox{for $j\le t$,} \label{22yi}\\
0=\frac{\partial \mathcal{L}}{\partial y_j} &= -2\lambda y_j -\mu
\qquad\hbox{for $2<j\le n$,} \label{22yj}\\
0=\frac{\partial \mathcal{L}}{\partial \lambda}&= x_1^2+x_2^2+y_1^2+\ldots+y_n^2 -1,\label{22lambda}\\
0=\frac{\partial \mathcal{L}}{\partial \mu}&= x_1+x_2+y_1+\ldots+y_n.\label{22mu}
\end{align}

As in the previous proof, we define $\mathcal{S}(D)$ as the set of all
$(f,\lambda,\mu)\in\mathbb{R}^{2+n}\times\mathbb{R}\times\mathbb{R}$
satisfying (\ref{22xi}), (\ref{22xj}), (\ref{22yj}), (\ref{22lambda}) and (\ref{22mu})
and $\mathcal{S}_{\lambda}(D):=\{\lambda:(f,\lambda,\mu)\in\mathcal{S}(D)\}$, $\lambda_{\max}:=\max\mathcal{S}_{\lambda}(D)$.
Then $\QEC(G)=-2+2\lambda_{\max}$.

From \eqref{22xi} we have
\begin{equation}\label{22y1y2}
y_1=(\lambda-1)x_1-x_2+\mu/2,\qquad
y_2=(\lambda-1)x_2-x_1+\mu/2,
\end{equation}
and substituting to \eqref{22xj} we get equalities
\begin{equation}\label{22x1x2}
2\lambda(1-\lambda)x_1+2(2\lambda-1)x_2=\lambda\mu,\qquad
2(2\lambda-1)x_1+2\lambda(1-\lambda)x_2=\lambda\mu.
\end{equation}
Now, applying \eqref{22yj} and \eqref{22y1y2} to \eqref{22mu} we get
\begin{equation}\label{22x1x2n4}
2\lambda(\lambda-1)(x_1+x_2)=(n-2-2\lambda)\mu.
\end{equation}

Assume that $(f,1,\mu)\in\mathcal{S}(D)$.
Then (\ref{22xi}), (\ref{22xj}) imply that $x_1=x_2=\mu/2$, $y_1=y_2=0$.
Now (\ref{22yj}) and (\ref{22mu}) imply $y_j=-\mu/2$ for $2<j\le n$, $n=4$ and $\mu=\pm1$.
Therefore $1\in\mathcal{S}_{\lambda}(D)$ if and only if $n=4$.

Now assume that $(f,\lambda,0)\in\mathcal{S}(D)$.
Then $\lambda\ne1$ and from \eqref{22mu}, \eqref{22y1y2} we have $x_1+x_2=0$.
Hence $x_2=-x_1$, $y_1=\lambda x_1$, $y_2=-\lambda x_1$, $x_1\ne0$.
Now \eqref{22xj} implies that $\lambda^2+\lambda-1=0$.
On the other hand, if $\lambda^2+\lambda-1=0$ then we can choose
$x_1$ so that $x_1^2(2+2\lambda^2)=1$, put $x_2=-x_1$, $y_1=\lambda x_1$, $y_2=-\lambda x_1$, $x_1\ne0$
and all equations \eqref{22xi}--\eqref{22mu} are satisfied.
Therefore $(-1\pm\sqrt{5})/2\in\mathcal{S}_{\lambda}(D)$.

Now assume that $(f,\lambda,\mu)\in\mathcal{S}(D)$, $\mu\ne0$, $\lambda\ne1$ and $\lambda^2+\lambda-1\ne0$.
Subtracting \eqref{22xj} for $i=1$ from \eqref{22xj} for $i=2$ and using \eqref{22y1y2} we get
$(\lambda^2+\lambda-1)(x_1-x_2)=0$, consequently, from \eqref{22x1x2n4} and \eqref{22y1y2},  
\[
x_1=x_2=\frac{(n-2-2\lambda)\mu}{4\lambda(\lambda-1)},\qquad
y_1=y_2=\frac{(n\lambda+4-2n)\mu}{4\lambda(\lambda-1)}.
\]
Then, \eqref{22xi}, \eqref{22mu} are satisfied. Then \eqref{22xj} is satisfied if and only if
\begin{equation}
(n+2)\lambda^2-(3n-4)\lambda+n-2=0,
\end{equation}
and one can verify that
\[
\frac{\sqrt{5}-1}{2}<\frac{3n-4+\sqrt{5n^2-24n+32}}{2(n+2)}
\]
for $n\ge3$. It remains to choose $\mu$ such that \eqref{22lambda} holds.
\end{proof}

Now we note that if $m=n$ or $t=m$ then Theorem~\ref{th:Ktmn} can be formulated in a simpler form.

\begin{corollary}\label{cor:ktmm}
If $1\le t\le m$, $m\ge2$, then
\begin{equation}\label{eq:corktmm}
\QEC(K_{m,m}^{t})=\frac{m-6+\sqrt{m^2+4m+4-8t}}{2},
\end{equation}
in particular $\QEC(K_{m,m}^{m})=m-4$ for $m\ge4$.

If $3\le m\le n$, $n\ge4$ then
\begin{equation}\label{eq:corkmmn}
\QEC(K_{m,n}^{m})=\frac{mn   -4 m-2 n+\sqrt{m^2 n^2+4n^2-4m^2 n}}{m+n}.
\end{equation}
\end{corollary}

\begin{proof}
If $n=m$ then
\[
P(\lambda)=m(m\lambda-m+t) \left(2\lambda^2-m\lambda+2\lambda-m+t\right)
\]
and for $t=m$ we get
\[
P(\lambda)=\lambda\left((m+n)\lambda^2+m(2-n)\lambda+m-n\right),
\]
which leads to \eqref{eq:corktmm} and \eqref{eq:corkmmn}.
\end{proof}

\begin{corollary}
The only almost complete bipartite graphs $K_{m,n}^{t}$ which are connected and of $QE$ class are the following:
$K_{1,n}$ (the ``star''), $n\ge1$, $K_{2,2}=C_{4}$,
$K_{2,2}^{1}=P_4$, $K_{2,3}^{1}$,
$K_{2,3}^{2}=P_5$, $K_{2,4}^{2}$, $K_{3,3}^{2}$, $K_{3,3}^{3}=C_6$, $K_{3,4}^{3}$, $K_{4,4}^{4}$ (the skeleton of the cube).
\end{corollary}

One can check that each connected graph $K_{m,n}^{t}$ which is not of $QE$ class
contains an isometric copy of $K_{2,3}$.

\section{Tanaka quintuple in a graph}

Now we are going to present a result of
Tanaka \cite{Tan}, who proved that in the case of bipartite graphs,
type $QE$ can be characterized by absence of some kind of 5-element subsets, see also \cite{RothWinkler}.

\begin{definition}
A sequence of vertices $(v_1,v_2,v_3,v_4,v_5)$ in a graph $G=(V,E)$ is called \textit{Tanaka quintuple}
if
\begin{equation}\label{tanaka1}
\{v_1,v_2\},\{v_3,v_4\}\in E,
\end{equation}
\begin{equation}\label{tanaka2}
d(v_1,v_3)=d(v_2,v_4)=d(v_1,v_4)-1=d(v_2,v_3)-1,
\end{equation}
\begin{equation}\label{tanaka3}
d(v_5,v_2)=d(v_5,v_1)+1,\qquad
d(v_5,v_3)=d(v_5,v_4)+1.
\end{equation}
\end{definition}

Now we present Theorem~3 in~\cite{Tan} in a form adapted to our terminology.
The first part was not stated in~\cite{Tan}, but it easily follows from the proof.

\begin{theorem}[Tanaka \cite{Tan}]
\label{pro:tanaka}
If a graph $G$ admits Tanaka quintuple then $G$ is of non-$QE$ class. 
In addition, if $G$ is a bipartite graph then $G$ is of $QE$ class
if and only if $G$ has no Tanaka quintuple.
\end{theorem}

Let us collect some properties.

\begin{lemma}\label{le:tanaka}
Assume that $(v_1,v_2,v_3,v_4,v_5)$ is a Tanaka quintuple in a graph $G$.

If $(v_1=a_0,a_1,\ldots,a_r=v_3)$, $(v_2=b_0,b_1,\ldots,b_r=v_4)$ are geodesics then
\[
\{a_0,a_1,\ldots,a_r\}\cap\{b_0,b_1,\ldots,b_r\}=\emptyset\quad\hbox{and}\quad
v_5\notin\{a_0,a_1,\ldots,a_r,b_0,b_1,\ldots,b_r\}.
\]

If $(v_5=e_0,e_1,\ldots,e_p=v_1)$, $(v_5=f_0,f_1,\ldots,f_q=v_4)$ are geodesics then
\[
v_4\notin\{e_0,e_1,\ldots,e_p\}\quad\hbox{and}\quad
v_1\notin\{f_0,f_1,\ldots,f_q\}.
\]
\end{lemma}

\begin{proof}
Assume that $(v_1=a_0,a_1,\ldots,a_r=v_3)$, $(v_2=b_0,b_1,\ldots,b_r=v_4)$ are geodesics
and that $a_i=b_j$ for some $0<i,j<r$. If $i<j$ then $d(v_1,v_4)\le i+r-j<r$, which is impossible,
similarly if $j<i$. If $i=j$ then $d(v_1,v_4)\le r$, which contradicts~(\ref{tanaka2}).
Therefore these geodesics must be disjoint.
If $v_5=a_i$, $0<i<r$, then $d(v_5,v_3)=r-i$, which, by (\ref{tanaka3}), implies that $d(v_5,v_4)=r-i-1$.
Consequently $d(v_1,v_4)=r-1$, which contradicts~(\ref{tanaka2}).

If $(v_5=e_0,e_1,\ldots,e_p=v_1)$ is a geodesics and $v_4=e_i$, $0<i<p$,
then $d(v_5,v_2)=i+r$, $d(v_5,v_1)=i+r+1$, which contradicts~(\ref{tanaka3}).

The remaining statements can be proved in a similar manner.
\end{proof}

\section{Theta graphs}\label{sec:tgwt}

From now on we are going to study a special family of graphs, which we will call \textit{theta graphs}.
Let $P_{n+1}=(x_0,\dots,x_n)$ denote the path graph, i.e., graph with $n+1$ vertices $\{x_0,\ldots,x_n\}$ and $n$ edges joining the consecutive vertices $\{x_{i-1},x_i\}$ ($i=1,\ldots,n$).
\begin{definition}\label{Deftheta}
    For given integers $1\leq \alpha,\beta,\gamma$, with at most one of them equal $1$,
we define a \textit{theta graph} $\Theta(\alpha,\beta,\gamma)$ by taking three path graphs 
$$
P_{\alpha+1}=  (x_0,\ldots,x_{\alpha}) ,\ 
P_{\beta+1}=   ( y_0,\ldots,y_\beta),\ 
P_{\gamma+1}= (z_0,\ldots,z_\gamma),\ 
$$
with the same endpoints $x_0=y_0=z_0$,  $x_{\alpha}=y_{\beta}=z_{\gamma}$ and the remaining vertices $x_j$ ($j=1,\dots,\alpha-1$), $y_j$  ($j=1,\dots,\beta-1$) and $z_{j}$ ($j=1,\dots,\gamma-1$) being mutually different.
\end{definition} 
Therefore $\Theta(\alpha,\beta,\gamma)$ has $\alpha+\beta+\gamma-1$ vertices
and $\alpha+\beta+\gamma$ edges, see Figure~\ref{fig:theta234}.
For example, $\Theta(2,2,2)=K_{3,2}$.

\begin{center}
\includegraphics[width=250pt]{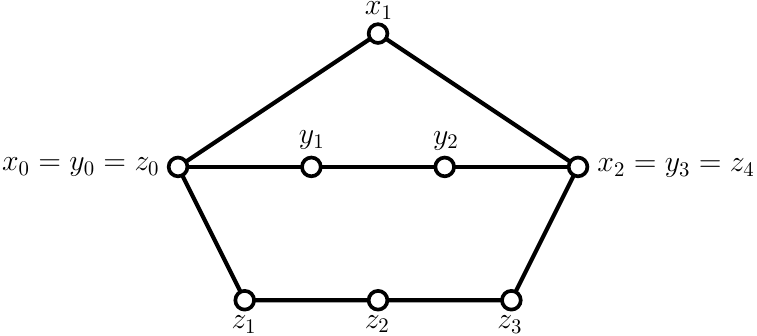}
\captionof{figure}{Theta graph $\Theta(2,3,4)$}\label{fig:theta234}
\end{center}

We are going to prove the following two results:

\begin{theorem}\label{thm:thetaqe}
Assume that $2\le\beta\le\gamma$. If either
\begin{itemize}
\item $\beta=2$, or
\item$\beta=3$, or
\item $\beta$ and $\gamma$ are odd,
\end{itemize}
then $\Theta(1,\beta,\gamma)$ is of $QE$ class.
\end{theorem}

\begin{theorem}\label{thm:thetanonqe}
Assume that $1\le\alpha\le\beta\le\gamma$, $\beta\ge2$. If either
\begin{itemize}
\item $\alpha=\beta=2$, or
\item $\alpha=2$, $\beta=3$, $\gamma$ is even, or
\item $\alpha=2$, $\beta\ge4$, or
\item $\alpha\ge3$,
\end{itemize}
then $\Theta(\alpha,\beta,\gamma)$ is of non-$QE$ class.
\end{theorem}

Note that for each $n\ge 5$ we can choose a triple $(\alpha,\beta,\gamma)$ such that $\alpha+\beta+\gamma-1=n$ and $\Theta(\alpha,\beta,\gamma)$ is of non-QE class.

\begin{corollary}\label{nonqeforeachn}
    For each $n\ge 5$ there exists a graph $\Theta(\alpha,\beta,\gamma)$ on $n$ vertices and of non-QE class. 
\end{corollary}

These results, as well as some particular examples, lead us to the following conjecture:

\begin{conjecture}\label{conj:thetaqe}
Assume that $1\le\alpha\le\beta\le\gamma$, $\beta\ge2$.
Then $\Theta(\alpha,\beta,\gamma)$ is of QE class if and only if either
\begin{itemize}
\item $\alpha=1$, or
\item $\alpha=2$, $\beta=3$ and $\gamma$ is odd.
\end{itemize}
\end{conjecture}

Once Theorems~\ref{thm:thetaqe} and \ref{thm:thetanonqe} are proved,
to complete the proof of this conjecture one needs to prove $QE$ property
of $\Theta(\alpha,\beta,\gamma)$ for the case when
\begin{itemize}
\item $\alpha=1$, $4\le\beta\le\gamma$ and $\beta$ or $\gamma$ is even, and
\item $\alpha=2$, $3=\beta\le\gamma$ and $\gamma$ is odd.
\end{itemize}

The proofs of both theorems, postponed to Section~\ref{se:proofs}, are largely based on Tanaka's theorem.
Therefore first we need to study these theta graphs which admit Tanaka quintuple, as well as modified
Tanaka quintuple, defined in Section~\ref{sec:tgwmtq}.
We start with a simple observation.

\begin{lemma}\label{le:thetabipartite}
The graph $\Theta(\alpha,\beta,\gamma)$ is bipartite if and only if $\alpha,\beta,\gamma$ are of the same parity.
\end{lemma}

\begin{proof}
Assume that $\alpha,\beta,\gamma$ have the same parity. We define $V_e$ (resp. $V_o$)
to be the set of all $x_i,y_j,z_k$ with $i,j,k$ even (resp. odd). By assumption, these sets are well defined,
disjoint, $V_e\cup V_o=V$ and no edge connects vertices either in $V_e$ or in $V_o$.

On the other hand, if for example, $\alpha$ and $\beta$ have different parity then $\Theta(\alpha,\beta,\gamma)$
contains a cycle of odd degree $\alpha+\beta$, consequently $\Theta(\alpha,\beta,\gamma)$ is not bipartite.
\end{proof}

Combining Theorems \ref{thm:thetaqe}, \ref{thm:thetaqe} and Lemma \ref{le:thetabipartite} we get the following corollary:
\begin{corollary}
Assume that $1\le\alpha\le\beta\le\gamma$, $\beta\ge2$. Then the graph $\Theta(\alpha,\beta,\gamma)$ is bipartite and of QE class if and only if $\alpha=1$ and $\beta, \gamma$ are odd.
\end{corollary}

\section{Theta graphs with Tanaka quintuple}\label{tgwt}

In this section we will characterize these theta graphs which admit Tanaka quintuple.

\begin{lemma}\label{le:theta}
Assume that  $(v_1,v_2,v_3,v_4,v_5)$ is a Tanaka quintuple in $\Theta(\alpha,\beta,\gamma)$,
with $v_5=z_k $ for some $0\le k\le\gamma$, i.e, $v_5$ is a vertex of $P_{\gamma+1}$, cf. Definition~\ref{Deftheta} for notation. Then
\begin{itemize}
\item[i)] $\alpha+\beta$ is even,
\item[ii)] $0<k<\gamma$,
\item[iii)] $v_1,v_2,v_3,v_4$ are vertices of $P_{\alpha+1}\cup P_{\beta+1}$,
but they can not be all vertices of $P_{\alpha+1}$ nor of $P_{\beta+1}$.
\end{itemize}
Moreover, if $v_1=x_{i}, v_2=x_{i+1}, v_3=y_j,v_4=y_{j+1}$ then
\begin{align}
2i+2j+2&=\alpha+\beta,\label{eq:1lemma}\\
2i+1&<\alpha+\gamma,\label{eq:2lemma}\\
2j+1&<\beta+\gamma.\label{eq:3lemma}
\end{align}
\end{lemma}

\begin{proof}
Let $(v_1=a_0,a_1,\ldots,a_r=v_3)$, $(v_2=b_0,b_1,\ldots,b_r=v_4)$ be geodesics.
Then, by Lemma~\ref{le:tanaka}, $\mathcal{C}:=(a_0,a_1,\ldots,a_r,b_0,b_1,\ldots,b_r)$ is a cycle in $\Theta(\alpha,\beta,\gamma)$
and $v_5$ is not a vertex of $\mathcal{C}$. This implies that $\mathcal{C}=P_{\alpha+1}\cup P_{\beta+1}$,
$\alpha+\beta=2r+2$ and $0<k<\gamma$.

Assume that $v_1,v_2,v_3,v_4$ are vertices of $P_{\alpha+1}$.
If $v_1=x_i$, $v_2=x_{i+1}$, $v_4=x_j$, $v_3=x_{j+1}$ then a geodesics from $v_5$ to $v_4$ must pass either through $v_1$
or through $v_3$, which is impossible by (\ref{tanaka3}) and the second part of Lemma~\ref{le:tanaka}.
Similarly, if $v_2=x_i$, $v_1=x_{i+1}$, $v_3=x_j$, $v_4=x_{j+1}$, then a geodesics from $v_5$ to $v_1$ must pass either through $v_2$
or through $v_4$, which, again, is impossible. This proves~(iii).

Now, assume that $v_1=x_{i}, v_2=x_{i+1}, v_3=y_j,v_4=y_{j+1}$. Then $(x_i,\ldots,x_0,y_1,\ldots,y_j)$ is a geodesic from $v_1$
to $v_3$, consequently $d(v_1,v_3)=i+j$. Similarly, $d(v_2,v_4)=\alpha-i+\beta-j-2$, which leads to (\ref{eq:1lemma}).
Inequality (\ref{eq:2lemma}) means that the walk
from $v_2$ to $v_3$ through $P_{\gamma+1}$:
\[
\left(v_2=x_{i+1},x_{i+2},\ldots,x_{\alpha}=z_{\gamma},z_{\gamma-1},\ldots,z_0=y_0,y_{1},\ldots,y_{j}=v_3\right),
\]
is longer that $i+j$, by~(\ref{tanaka2}). Similarly for~(\ref{eq:3lemma}).
\end{proof}

\begin{lemma}\label{le:tanaka2}
If $\Theta(\alpha,\beta,\gamma)$ admits a Tanaka quintuple then $\alpha,\beta,\gamma\ge2$.
\end{lemma}

\begin{proof}
This is a consequence of part (ii) and (iii) of Lemma~\ref{le:theta}.
\end{proof}

\begin{lemma}\label{le:tanaka3}
If $\alpha,\beta\ge3$ are odd then $\Theta(\alpha,\beta,2)$ does not admit Tanaka quintuple.
\end{lemma}

\begin{proof}
Assume, that $\alpha=2k+1$, $\beta=2l+1$, cf. Definition~\ref{Deftheta} for notation.
If $(v_1,v_2,v_3,v_4,v_5)$ is a Tanaka quintuple in $\Theta(\alpha,\beta,2)$ then, in view of Lemma~\ref{le:tanaka},
$v_5=z_1$ and we can assume that
\[
v_1=x_i,\quad v_2=x_{i+1},\quad v_3=y_j,\quad v_4=y_{j+1},
\]
$0\le i\le 2k$, $0\le j\le 2l$.
Then (\ref{eq:1lemma}), (\ref{eq:2lemma}), (\ref{eq:3lemma}) imply that $i=k$, $j=l$. But then
\[
d(v_1,v_5)=d(v_2,v_5)=k+1,\qquad
d(v_3,v_5)=d(v_4,v_5)=l+1,
\]
which contradicts (\ref{tanaka3}).
\end{proof}

\begin{lemma}
If $\alpha,\beta$ are even then $\Theta(\alpha,\beta,3)$ does not admit Tanaka quintuple.
\end{lemma}

\begin{proof}
Put $\alpha=2k$, $\beta=2l$, cf. Definition~\ref{Deftheta} for notation. Then we can assume that
\[
v_1=x_i,\quad v_2=x_{i+1},\quad v_3=y_j,\quad v_4=y_{j+1},
\]
$0\le i\le 2k-1$, $0\le j\le 2l-1$. Then (\ref{eq:1lemma}), (\ref{eq:2lemma}), (\ref{eq:3lemma}) imply,
that either $i=k, j=l-1$ or~$i=k-1, j=l$.
If holds the former then $d(v_1,z_1)=d(v_2,z_1)=k+1$ and $d(v_1,z_2)=k+1>d(v_2,z_2)=k$,
hence it is impossible to choose $v_5$ satisfying~(\ref{tanaka3}). Similarly in the latter case.
\end{proof}

\begin{lemma}
If $\beta\ge2$ is even and $\gamma\ge3$ is odd then $\Theta(2,\beta,\gamma)$ does not admit Tanaka quintuple.
\end{lemma}

\begin{proof}
Put $\beta=2l, \gamma=2p+1$ and assume that $(v_1,v_2,v_3,v_4,v_5)$ is a Tanaka quintuple in $\Theta(2,\beta,\gamma)$, cf. Definition~\ref{Deftheta} for notation.
By symmetry we can assume, that $v_1=x_0$, $v_2=x_1$, $v_3=y_j$, $v_4=y_{j+1}$, $v_5=z_k$.
Then, in view of (\ref{eq:1lemma}), $j=l$. Now,
\[
d(v_5,v_1)=k,\quad
d(v_5,v_2)=2p+2-k,
\]
which, by (\ref{tanaka3}) implies $k<p+1$, and
\[
d(v_5,v_3)=k+l,\quad
d(v_5,v_4)=2p-k+l,
\]
which implies $p<k$, which is impossible.
\end{proof}

\begin{proposition}
Assume that $\alpha,\beta,\gamma\ge2$.
If either
\begin{itemize}
\item $\alpha=2$, $\beta,\gamma$ are even, or
\item $\alpha,\beta\ge4$ are even, $\gamma\ne3$, or
\item $\alpha,\beta\ge3$ are odd, $\gamma\ge3$,
\end{itemize}
then $\Theta(\alpha,\beta,\gamma)$ admits Tanaka quintuple, with
$v_5\in P_{\gamma+1}$, cf. Definition~\ref{Deftheta} for notation.

In particular, if $\alpha,\beta,\gamma\ge4$ then $\Theta(\alpha,\beta,\gamma)$
admits Tanaka quintuple.
\end{proposition}

\begin{proof}
First assume that $\alpha=2l+2$, $\beta=\alpha+2k=2k+2l+2$, $k,l\ge0$.
If $\gamma=2p$ and  either $l=0$, $p\ge1$ or $p=1$, $l\ge1$ then we put
\[
v_1:=x_{l},\quad v_2:=x_{l+1},\quad v_3:=y_{k+l+1},\quad v_4:=y_{k+l+2},\quad v_5:=z_{p},
\]
and if $l\ge1$, $\gamma=2p+\epsilon\ge4$, with $p\ge1$, $\epsilon\in\{0,1\}$ we set
\[
v_1:=x_{l-1},\quad v_2:=x_{l},\quad v_3:=y_{k+l+2},\quad v_4:=y_{k+l+3},\quad v_5:=z_{p}.
\]
It is easy to verify that $(v_1,v_2,v_3,v_4,v_5)$ is a Tanaka quintuple.
Note that for $\gamma=3$ we have $d(v_1,v_3)=d(v_1,v_4)=k+2l+4$, so in this case
the construction doesn't work.

Now assume that $\alpha=2l+3$, $l\ge0$, $\beta=\alpha+2k=2k+2l+3$ and $\gamma=2p+\epsilon$,
with $p\ge1$, $\epsilon\in\{0,1\}$. Then we put
\[
v_1:=x_{l},\quad v_2:=x_{l+1},\quad v_3:=y_{k+l+2},\quad v_4:=y_{k+l+3},\quad v_5:=z_p,
\]
and, again, $(v_1,v_2,v_3,v_4,v_5)$ becomes a Tanaka quintuple for $\gamma\ge3$.
\end{proof}

Summing up, we have the following result.

\begin{theorem}\label{th:thetawith5t}
Assume, that $1\le\alpha\le\beta\le\gamma$, $\beta\ge2$. Then $\Theta(\alpha,\beta,\gamma)$
admits Tanaka quintuple if and only if either
\begin{itemize}
\item $\alpha=2$ and $\beta,\gamma$ are even, or
\item $\alpha=3$ and $\beta,\gamma$ are odd, or
\item $\alpha\ge4$.
\end{itemize}
\end{theorem}

\section{Theta graphs with modified Tanaka quintuple}\label{sec:tgwmtq}

We make a slight change for Tanaka quintuple, let $d(v_3,v_5)=d(v_4,v_5)$. A sequence of vertices $(v_1,v_2,v_3,v_4,v_5)$
in a graph $G=(V,E)$ is called \textit{modified Tanaka quintuple} if:
\begin{equation}\label{tanaka4}
\{v_1,v_2\},\{v_3,v_4\}\in E,
\end{equation}
\begin{equation}\label{tanaka5}
d(v_1,v_3)=d(v_2,v_4)=d(v_1,v_4)-1=d(v_2,v_3)-1,
\end{equation}
\begin{equation}\label{tanaka6}
d(v_5,v_2)=d(v_5,v_1)+1,\qquad
d(v_5,v_3)=d(v_5,v_4).
\end{equation}

Note that only non-bipartite graphs may have a modified Tanaka quintuple,
because vertices $v_3$ and $v_4$ are adjacent and have the same distance to $v_5$.

\begin{theorem}\label{modifiedtanaka}
If a graph $G$ admits modified Tanaka quintuple then $G$ is of non-$QE$ class.
\end{theorem}

\begin{proof}
Let $(v_1,v_2,v_3,v_4,v_5)$ be a modified Tanaka quintuple in a graph $G$ and assume that
$$
d(v_1,v_3)=d(v_2,v_4)=i, \hspace{0.2cm}d(v_1,v_5)=j \hspace{0.2cm} \mathrm{and} \hspace{0.2cm}d(v_4,v_5)=d(v_3,v_5)=h.
$$
Then
$$
d(v_2,v_3)=i+1, \hspace{0.2cm} \mathrm{and} \hspace{0.2cm} d(v_2,v_5)=j+1.
$$
Set
$$
f(v_1)=f(v_4)=-(j+h), \hspace{0.2cm} f(v_2)=j+h, \hspace{0.2cm} f(v_3)=j+h-1, \hspace{0.2cm} f(v_5)=1.
$$
Then
\begin{multline*}
\sum_{v_i,v_j\in V}f(v_i) f(v_j) d(v_i,v_j)=\\-(j+h)(j-i)+(j+h)(j-i)+(j+h-1)h-(j+h)(h-1)+j=2j >0,
\end{multline*}
hence $G$ is of non-$QE$ class.
\end{proof}

Now we observe that for the modified Tanaka quintuple the analog of Lemma~\ref{le:tanaka} holds.

\begin{lemma}\label{le:tanakamodified}
Assume that $(v_1,v_2,v_3,v_4,v_5)$ is a modified Tanaka quintuple in a graph~$G$.

If $(v_1=a_0,a_1,\ldots,a_r=v_3)$, $(v_2=b_0,b_1,\ldots,b_r=v_4)$ are geodesics then
\[
\{a_0,a_1,\ldots,a_r\}\cap\{b_0,b_1,\ldots,b_r\}=\emptyset\quad\hbox{and}\quad
v_5\notin\{a_0,a_1,\ldots,a_r,b_0,b_1,\ldots,b_r\}.
\]

If $(v_5=e_0,e_1,\ldots,e_p=v_1)$, $(v_5=f_0,f_1,\ldots,f_q=v_4)$, $(v_5=g_0,g_1,\ldots,g_q=v_3)$ are geodesics then
\[
v_3,v_4\notin\{e_0,e_1,\ldots,e_p\},\quad\hbox{and}\quad
v_1\notin\{f_0,f_1,\ldots,f_q\} \cup \{g_0,g_1,\ldots,g_q\}.
\]
\end{lemma}

\begin{proof}
If $(v_5=e_0,e_1,\ldots,e_p=v_1)$ is a geodesics and $v_3=e_i$, $0<i<p $, then $d(v_5,v_1)=d(v_5,v_2)$, which contradicts (\ref{tanaka6}).
The rest of the proof is similar to that of Lemma \ref{le:tanaka}.
It is sufficient to replace $d(v_3,v_5)=d(v_4,v_5)-1$ by $d(v_3,v_5)=d(v_4,v_5)$.
\end{proof}

In a similar manner we can prove that Lemma~\ref{le:theta}, \ref{le:tanaka2} and \ref{le:tanaka3}
are also true for modified Tanaka quintuple. In the proofs, it is sufficient to replace $d(v_3,v_5)=d(v_4,v_5)-1$ by $d(v_3,v_5)=d(v_4,v_5)$.

\begin{lemma}\label{le:thetamodified}
Assume that  $(v_1,v_2,v_3,v_4,v_5)$ is a Tanaka quintuple on $\Theta(\alpha,\beta,\gamma)$,
with $v_5=z_k$ for some $0\le k\le\gamma$, cf. Definition~\ref{Deftheta} for notation. Then
\begin{itemize}
\item[i)] $\alpha+\beta$ is even,
\item[ii)] $0<k<\gamma$,
\item[iii)] $v_1,v_2,v_3,v_4$ are vertices of $P_{\alpha+1}\cup P_{\beta+1}$,
but they can not be all vertices of $P_{\alpha+1}$ nor of $P_{\beta+1}$
\end{itemize}
Moreover, if $v_1=x_{i}, v_2=x_{i+1}, v_3=y_j,v_4=y_{j+1}$ then
\begin{align}
2i+2j+2&=\alpha+\beta,\label{eq:1lemmamodified}\\
2i+1&<\alpha+\gamma,\label{eq:2lemmamodified}\\
2j+1&<\beta+\gamma.\label{eq:3lemmamodified}
\end{align}
\end{lemma}

\begin{lemma}\label{le:tanaka2modified}
If $\Theta(\alpha,\beta,\gamma)$ admits a modified Tanaka quintuple then $\alpha,\beta,\gamma\ge2$.
\end{lemma}

\begin{lemma}
If $\alpha,\beta\ge3$ are odd then $\Theta(\alpha,\beta,2)$ does not admit Tanaka quintuple.
\end{lemma}

Now we are able to characterize theta graphs with modified Tanaka quintuple.

\begin{theorem}\label{thetawithmodifiedtanaka}
    Assume that $\alpha,\beta,\gamma\ge 2$. If either
    \begin{itemize}
        \item $\alpha,\beta$ are even, $\gamma$ is odd, or
        \item $\alpha,\beta$ are odd, $\gamma\ge 4$ is even,
    \end{itemize}
    then $\Theta(\alpha,\beta,\gamma)$ admits modified Tanaka quintuple, with $v_5\in P_{\gamma+1}$, cf. Definition~\ref{Deftheta} for notation.
\end{theorem}
\begin{proof}
    First assume that $\alpha=2k$, $\beta=2l$, $k,l\ge 1$. If $\gamma=2p+1$, $p\ge 1$, then we put
    \[
v_1:=x_{k},\quad v_2:=x_{k+1},\quad v_3:=y_{l},\quad v_4:=y_{l+1},\quad v_5:=z_{p}.
\]
    It is easy to verify that $(v_1,v_2,v_3,v_4,v_5)$ is a modified Tanaka quintuple.

    Now assume that $\alpha=2k+1$, $\beta=2l+1$, $k,l\ge 1$. If $\gamma=2p+2$, $p\ge 1$, then we put
    \[
v_1:=x_{k-1},\quad v_2:=x_{k},\quad v_3:=y_{l+1},\quad v_4:=y_{l+2},\quad v_5:=z_{p},
\]
and, again $(v_1,v_2,v_3,v_4,v_5)$ becomes a modified Tanaka quintuple.
\end{proof}

\section{Proofs of Theorem~\ref{thm:thetaqe} and Theorem~\ref{thm:thetanonqe}}\label{se:proofs}

Now we are ready to prove Theorem~\ref{thm:thetaqe} and Theorem~\ref{thm:thetanonqe}.

\begin{proof}[Proof of Theorem~\ref{thm:thetaqe}.]
First assume that $\alpha=1$, $\beta=2$, $\gamma=n-1\ge2$. Then
\[
d(z_i,z_j)=\min\{|i-j|,n-|i-j|\},\qquad
d(z_i,y_1)=\min\{i+1,n-i\}.\]
Let $f:V\to\mathbb{R}$ be a function satisfying $\sum_{x\in V}f(x)=0$.
Then putting $f(z_i)=u_{i+1}$ we have
\begin{multline*}
\sum_{x,y\in V}d(x,y)f(x)f(y)=\sum_{i,j=1}^{n}\min\{|i-j|,n-|i-j|\}u_i u_j\\
-\sum_{i=1}^{n}\min\{i,n+1-i\}u_i u_n
-\sum_{j=1}^{n}\min\{j,n+1-j\}u_n u_j
=-\sum_{i,j=1}^{n} a^{n}_{i,j}u_i u_j,
\end{multline*}
where
\begin{equation}\label{eq:theta12a}
a^{n}_{i,j}=\min\{i,n+1-i\}+\min\{j,n+1-j\}-\min\{|i-j|,n-|i-j|\}.
\end{equation}
Now it suffices to prove that the matrix $A_n:=(a^{n}_{i,j})_{i,j=1}^{n}$ is (weakly) positive definite.
For example:
\[
A_6=
\left(
\begin{array}{cccccc}
 2 & 2 & 2 & 1 & 1 & 1 \\
 2 & 4 & 4 & 3 & 1 & 1 \\
 2 & 4 & 6 & 5 & 3 & 1 \\
 1 & 3 & 5 & 6 & 4 & 2 \\
 1 & 1 & 3 & 4 & 4 & 2 \\
 1 & 1 & 1 & 2 & 2 & 2
\end{array}
\right),\qquad
A_7=\left(
\begin{array}{ccccccc}
 2 & 2 & 2 & 2 & 1 & 1 & 1 \\
 2 & 4 & 4 & 4 & 2 & 1 & 1 \\
 2 & 4 & 6 & 6 & 4 & 2 & 1 \\
 2 & 4 & 6 & 8 & 6 & 4 & 2 \\
 1 & 2 & 4 & 6 & 6 & 4 & 2 \\
 1 & 1 & 2 & 4 & 4 & 4 & 2 \\
 1 & 1 & 1 & 2 & 2 & 2 & 2
\end{array}
\right).
\]

For $1\le p\le q\le n$ put $\mathbf{J}^n(p,q)=\left(\mathbf{J}^{(p,q)}_{i,j}\right)_{i,j=1}^{n}$, where $\mathbf{J}^{(p,q)}_{i,j}=1$
if $p\le i,j\le q$, and $\mathbf{J}^{(p,q)}_{i,j}=0$ otherwise. Note that each $\mathbf{J}^n(p,q)$ is positive definite.

One can check that
\[
A_6=\mathbf{J}^{6}(1,6)+\mathbf{J}^{6}(1,3)+\mathbf{J}^{6}(4,6)+2\mathbf{J}^{6}(2,4)+2\mathbf{J}^{6}(3,5),
\]
\begin{multline*}
\phantom{aaaaaaaaa}
A_7=\mathbf{J}^{7}(1,7)+\mathbf{J}^{7}(1,4)+\mathbf{J}^{7}(2,5)+\mathbf{J}^{7}(3,6)+\mathbf{J}^{7}(4,7)\\
+\mathbf{J}^{7}(2,4)+\mathbf{J}^{7}(3,5)+\mathbf{J}^{7}(4,6).\phantom{aaaaaa}
\end{multline*}
In general, we claim that
\begin{equation}\label{eq:theta12even}
A_{2k}=\mathbf{J}^{2k}(1,2k)+\mathbf{J}^{2k}(1,k)+\mathbf{J}^{2k}(k+1,2k)+2\sum_{p=2}^{k}\mathbf{J}^{2k}(p,p+k-1),
\end{equation}
\begin{equation}\label{eq:theta12odd}
A_{2k+1}=\mathbf{J}^{2k+1}(1,2k+1)+\sum_{p=1}^{k+1}\mathbf{J}^{2k+1}(p,p+k)+\sum_{p=2}^{k+1}\mathbf{J}^{2k+1}(p,p+k-1).
\end{equation}

Denote the right hand side of (\ref{eq:theta12even},\ref{eq:theta12odd}) by $B_n=\left(b^{n}_{i,j}\right)_{i,j=1}^{n}$
and observe that
\begin{equation}\label{eq:theta12sym}
a^{n}_{i,j}=a^{n}_{j,i}=a^{n}_{n+1-i,n+1-j},\qquad
b^{n}_{i,j}=b^{n}_{j,i}=b^{n}_{n+1-i,n+1-j}
\end{equation}

Let $n=2k$. If $1\le i,j\le k$ then $a^{2k}_{i,j}=2\min\{i,j\}=b^{2k}_{i,j}$.
Now assume that $1\le i\le k$, $k+1\le j\le 2k$. Then
\[
a^{2k}_{i,j} 
=\left\{\begin{array}{ll}
1&\hbox{if $k\le j-i$,}\\
2k+1+2i-2j&\hbox{otherwise.}
\end{array}\right.
\]
On the other hand note that $\mathbf{J}^{(p,p+k-1)}_{i,j}=1$ if and only if $j-k<p\le i$,
hence on the right hand side of (\ref{eq:theta12even}) we obtain
\[
b^{2k}_{i,j}=1+2\max\{0,k+i-j\}
\]
and this equals $=a^{2k}_{i,j}$.
The remaining cases follow from (\ref{eq:theta12sym}), therefore $A_{2k}=B_{2k}$.

Now assume that $n=2k+1$.
Assume that $1\le i,j\le k+1$, $i+j\le 2k+1$.
Then $a^{2k+1}_{i,j}=2\min\{i,j\}=b^{2k+1}_{i,j}$.

Now assume that $1\le i\le k+1$, $k+2\le j\le 2k+1$.
\[
a^{2k+1}_{i,j} 
=\left\{\begin{array}{ll}
1&\hbox{if $k< j-i$,}\\
2k+2+2i-2j&\hbox{otherwise.}
\end{array}\right.
\]
On the right hand side of (\ref{eq:theta12odd}) we obtain
\[
b^{2k+1}_{i,j}=1+\max\{0,k+i-j+1\}+\max\{0,k+i-j\},
\]
which is exactly $a^{2k+1}_{i,j}$.
Again we apply (\ref{eq:theta12sym}) to conclude that $A_{2k+1}=B_{2k+1}$.

Summing up, for each $n\ge3$ the matrix $A_n$ is positive definite as a sum of positive definite matrices.
This concludes this part of proof.

For $\alpha=1$, $\beta=3$ the statement is a consequence of Proposition~3.15 in \cite{ObataPrim6v}.

Finally, if $\alpha=1$ and $\beta,\gamma$ are odd then the statement is a consequence
of Lemma~\ref{le:thetabipartite}, Theorem~\ref{pro:tanaka} and Theorem~\ref{th:thetawith5t}.
\end{proof}

For the proof of Theorem~\ref{thm:thetanonqe} we need a lemma.

\begin{lemma}\label{2oddodd}
If $k,l\ge2$ then $\Theta(2,2k+1,2l+1)$ is of non-$QE$ class.
\end{lemma}

\begin{proof}
Put
\[
f(y_1)=f(z_{2l})=u_1,\quad
f(y_{k+1})=f(z_{l})=u_2,\quad
f(y_{k+2})=f(z_{l-1})=u_3
\]
\[
f(x_1)=u_4=-2u_1-2u_2-2u_3,
\]
with
\[
u_1=2k+2l-4,\quad
u_2=11-4k-4l,\quad
u_3=4k+4l-13,
%\quad u_4=12-4k-4l,
\]
and $f(x)=0$ for all other vertices. Then
\begin{multline*}
\sum_{x,y\in V}d(x,y)f(x)f(y)=2u_1\big(4u_1+(k+l+1)u_2+(k+l+1)u_3+2u_4\big)\\
+u_2\big(2(k+l+1)u_1+2(k+l+1)u_2+2(k+l+1)u_3+(k+l+2)u_4\big)\\
+u_3\big(2(k+l+1)u_1+2(k+l+1)u_2+2(k+l)u_3+(k+l)u_4\big)\\
+u_4\big(4u_1+(k+l+2)u_2+(k+l)u_3\big)\\
=2(k+l-3)u_1\big(u_2+u_3\big)
-2u_2\left(u_1+u_2+u_3\right)+2u_3\big(u_1+u_2\big)\\
-2(u_1+u_2+u_3)\big(4u_1+(k+l+2)u_2+(k+l)u_3\big)\\
=-8(k+l-3)(k+l-2)+4(4k+4l-11)(k+l-3)\\
-2(4k+4l-13)(2k+2l-7)+8(k+l-3)^2\\
=(k+l-3)(16k+16l-52)-2(4k+4l-13)(2k+2l-7)\\
=2(4k+4l-13)\ge6>0.
\end{multline*}
This concludes the proof.
\end{proof}

\begin{proof}[Proof of Theorem~\ref{thm:thetanonqe}.]
Assume that $1\le\alpha\le\beta\le\gamma$, $\beta\ge2$.

If either $\alpha=2$, $\beta,\gamma$ are even, or $\alpha=3$, $\beta,\gamma$ are odd, or $\alpha\ge4$
then, by Theorem~\ref{th:thetawith5t}, $\Theta(\alpha,\beta,\gamma)$ admits a Tanaka quintuple,
hence is of non-$QE$ class in view of the first part of Theorem~\ref{pro:tanaka}.
If $\alpha=3$ and $\beta$ is even or $\gamma$ is even then, by Theorem~\ref{thetawithmodifiedtanaka}, $\Theta(\alpha,\beta,\gamma)$
admits a modified Tanaka quintuple, and hence, by Theorem~\ref{modifiedtanaka} is of non-$QE$ class.
Similarly, if $\alpha=2$, $\beta$ is even, $\gamma$ is odd, or  $\alpha=2$, $\beta$ is odd, $\gamma$ is even.
The remaining case, when $\alpha=2$ and $\beta,\gamma\ge5$ are odd, is covered by Lemma~\ref{2oddodd}.
\end{proof}

\begin{example}\label{examples}
Let us record some known examples:
\begin{align*}
\mathrm{QEC}\big(\Theta(1,2,2)\big)&=\frac{-1}{2},\\ 
\mathrm{QEC}\big(\Theta(1,2,3)\big)&=0,\\ 
\mathrm{QEC}\big(\Theta(2,2,2)\big)&=\frac{2}{5},\\ 
\mathrm{QEC}\big(\Theta(1,2,4)\big)&=\lambda_0=-0.3121\ldots,\
\mathrm{QEC}\big(\Theta(1,3,3)\big)&=0,\\ 
\mathrm{QEC}\big(\Theta(2,2,3)\big)&=\frac{\sqrt{19}-4}{3}>0,\\ 
\end{align*}
where $\lambda_0$ is the largest root of the equation $3\lambda^3+15\lambda^2+14\lambda+3=0$
(see $K_4\setminus\{e\}$, $K_4\setminus P_5$ in \cite{Obata3},
and G5-10, G6-32, G6-35, G6-30 in \cite{ObataPrim6v}). 
\end{example}

\begin{example}\label{example2}
    We present the values of $\QEC$ for theta graphs on 7, 8, 9 and 10 vertices, which are not mentioned in Theorems~\ref{thm:thetaqe} and \ref{thm:thetanonqe}, see Table \ref{78910vertices}. Note that the values are non-positive, thus Conjecture~\ref{conj:thetaqe} is true for 7, 8, 9 and 10 vertices.  
\begin{table}[h]
        \centering
        \begin{tabular}{|c|c|c|}\hline
         $\QEC(\Theta(2,3,3))$ & 
         $\QEC(\Theta(1,4,4))$&
         $\QEC(\Theta(1,4,5))$  
           \\ \hline
           0 &$\approx -0.1569$ & 0 \\\hline
           \hline
          $\QEC(\Theta(2,3,5))$ &
          $\QEC(\Theta(1,4,6))$&
           $\QEC(\Theta(1,5,5))$
          \\ \hline
          0&$\approx -0.1240$ &
          0
          
          \\ \hline
        \end{tabular}
        \caption{$\QEC$ of theta graphs on 7, 8, 9 and 10 vertices, which are not mentioned in Theorems~\ref{thm:thetaqe} and \ref{thm:thetanonqe}} 
        \label{78910vertices}
    \end{table}

\end{example}

\section{Answer to the main question}\label{spnonqe}

A connected graph is called a \textit{primary non-$QE$} if it is of non-QE class and every isometrically embedded connected
subgraph of $G$ if of QE-type.
Some examples of primary non-QE graphs were studied by Obata \cite{ObataPrim6v}.
Now we will see that there are infinitely many of them.
Namely, we observe that every theta graph which is not of QE-type is primary non-QE in a strong sense.

\begin{theorem}\label{thm:primarynonqe}
If $H$ is a proper connected subgraph of a theta graph
then $H$ is of QE type.
Consequently, every theta graph of non-QE type is primary non-QE.
\end{theorem}

\begin{proof}Observe that each subgraph of theta graph is either  a cycle, or a path, or a star products of 2 paths, or a  star product of a cycle and a path, or a star product of a cycle and two disjoint paths. All such  graphs are of QE-class, see \cite{Jak2, Mlo1, Obata3}.  
\end{proof}

Combining Theorem~{5.2} with Theorem~\ref{thm:primarynonqe} we obtain

\begin{corollary}\label{cor:primaryinfinite}
The class of all primary non-QE graphs is infinite. 
\end{corollary}
For closing the paper observe that every theta graph of non-QE type has a property that is even stronger than being primary non-QE. Namely,  
it contains no proper connected subgraph of non-QE class. Here by a subgraph of $G$ we understand a  graph $H$, such that the sets of vertices and edges of $H$ are contained in the sets of vertices and edges of $G$, respectively.

It appears that the theta graphs are the only bipartite graphs enjoying this property. The corresponding theorem is preceded by definition of \emph{subdivided wheel} and an example of a graph which is of QE class and has  subgraphs of type $\Theta(\alpha,\beta,\gamma)$ of non-QE class (hence, neither of  of them is isometrically embedded in $G$).

\begin{example}\label{subdivided}
First we define graphs $W_k$ and $W_k(m_1,\ldots,m_l;n_1,\ldots,n_k)$. The \emph{wheel} $W_k$ is the graph obtained as a join the graph $K_1$ and the cycle $C_k$. Let $u$ be the central vertex and $w_1,\ldots,w_k$ are the vertices of the remaining vertices of $W_k$. Let $W_k(m_1,\ldots,m_l;n_1,\ldots,n_k)$ be the graph obtaining by subdividing edges of $W_k$, where $m_i$ and $n_i$ are the number of vertices added on the edge $w_i w_{i+1}$ and $uw_i$, respectively.
Glavier, Klavzar and Mollard proved in \cite{Klavzar}  that for 
$k\ge3$ the graph $W_k(m_1,\ldots,m_l;$ $n_1,\ldots,n_k)$ is isometrically embeddable into a hypercube if and only if $m_i$ is odd for $i=1,\ldots,k$ and $n_1=\ldots n_k=0$, or $k=3$ and $m_1=m_2=m_3=n_1=n_2=n_3=1$. 

Consider now the graph $W_3(1,1,1;1,1,1)$ in Figure \ref{W3}. By result above it is isometrically embeddable into a hypercube, hence is of QE class. However, note that  it contains six subgraphs $\Theta(2,4,4)$.

\begin{figure}[h]
\includegraphics[width=150pt]{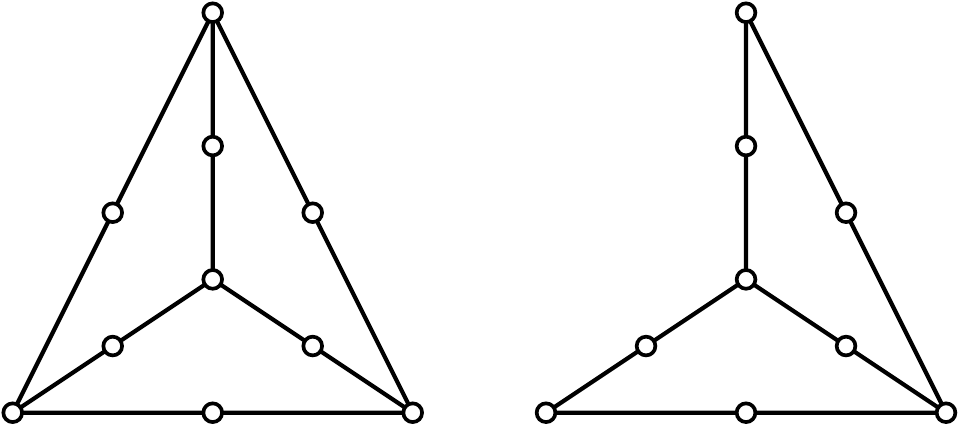}
\captionof{figure}{The graph $W_3(1,1,1;1,1,1)$ and its subgraph $\Theta(2,4,4)$}\label{W3}
\end{figure}
\end{example}

\begin{theorem}
    Let $G$ be a bipartite graph. Then the following are equivalent:
    \begin{enumerate}[\rm (i)]
\item $G$ is a non-QE graph and contains no proper connected subgraph   of non-QE class;
\item $G$ is a theta graph $\Theta(\alpha,\beta,\gamma)$, where $\alpha,\beta,\gamma \ge2$ are all even or all odd.
\end{enumerate}
\end{theorem}
\begin{proof}
   
The implication (ii)$\Rightarrow$(i) was already shown in Theorem \ref{thm:primarynonqe}, let us consider a bipartite graph $G$ satisfying (i) and suppose that it is not a theta graph. By Theorem \ref{pro:tanaka},  $G$ has a Tanaka quintuple $(v_1,v_2,v_3,v_4,v_5)$.  

Let $X:=(v_1=x_1,\ldots,x_m=v_3)$ and $Y:=(v_2=y_1,\ldots,y_m=v_4)$ be the geodesics. 
Clearly $G$ needs to have some extra edges, i.e., edges that are neither in $X$ nor in $Y$, and clearly none of these edges cannot connect vertices, neither from $X$, nor from $Y$ as this would violate the definition of geodesics.  We analyse now several instances, obtaining each time a contradiction. 

Assume that there is an edge other than $v_1v_2$ and $v_3v_4$ that connects a vertex from $X$ with vertex from $Y$, then by Lemma \ref{le:tanaka} we still do not have a Tanaka quintuple. 
Therefore, for the existence of Tanaka quintuple, there must exist a path $Z:=(z_1,\ldots,z_n)$ with length at least 2, such that $Z\cap (X\cup Y)=\{z_1,z_n\}$ and $v_5\in Z$. Note that vertices $z_1$ and $z_n$ must be adjacent. Indeed, otherwise  $G$ contains properly $\Theta(\alpha,\beta,\gamma)$, where $\alpha,\beta,\gamma \ge 2$ and $\alpha,\beta,\gamma$ are all even or all odd and consequently,  by Theorem~\ref{thm:thetanonqe}, $G$ contains a subgraph of non-QE class, contradiction.

Since the graph $G$ is bipartite and vertices $z_1$ and $z_n$ are adjacent, the path $Z$ has a length at least 3. Without loss of generality, assume that $z_1,z_n\in X\cup \{v_2\}$. To satisfy the condition $d(v_5,v_3)=d(v_5,v_4)+1$, the graph $G$ must contain a path $W:=(w_1,\ldots,w_i)$ such that $w_1\in Y$ and $w_i \in Z$. Clearly, $W$ and $X$ are disjoint, moreover $W$ has at least one edge that belongs neither to $Y$ nor to $Z$. We have $w_i\ne z_1$ and $w_i\ne z_n$, otherwise $d(v_1,v_3)=d(v_1,v_4)+1$, which is a contradiction with existence of Tanaka quintuple. 

Let $w_i=z_j$, where $j\in \{2,\ldots,n-1\}$, define the paths $Z_1:=(z_1,\ldots,z_j)$ and $Z_2:=(z_j,\ldots,z_n)$. Then  a subdivided wheel of non-QE class is a subgraph of $G$, see Example~\ref{subdivided}. Observe that either $Z_1$ or $Z_2$ has  length at least 2, without loss of generality assume the latter. Then $G$ has  a subgraph  $\Theta(\alpha,\beta,\gamma)$ with paths $P_\alpha:=Z_1 \cup \{z_n\}$, $P_\beta:=Z_2$ and $P_\gamma:=W\cup \{w_1,\ldots,v_4,v_3,\ldots, z_n\}$. Observe that  by Theorem~\ref{thm:thetanonqe} the graph $\Theta(\alpha,\beta,\gamma)$ is of non-QE class, which  is a  contradiction finishing the proof.
\end{proof}

\bibliographystyle{plain}
\bibliography{grafyMSW}

\begin{thebibliography}{10}

\bibitem{Alfakih}
A.~Alfakih.
\newblock {\em Euclidean Distance Matrices and Their Applications in Rigidity Theory}.
\newblock Springer, Cham., 2018.

\bibitem{Aro2019}
M.~Arockiaraj, S.~Klav{\v{z}}ar, J.~Clement, Sh. Mushtaq, and K.~Balasubramanian.
\newblock Edge distance-based topological indices of strength-weighted graphs and their application to coronoid systems, carbon nanocones and sio2 nanostructures.
\newblock {\em Molecular Informatics}, 38(11-12):1900039, 2019.

\bibitem{Asr1998}
A.S. Asratian, T.~Denley, and R.~H{\"a}ggkvist.
\newblock {\em Bipartite graphs and their applications}, volume 131.
\newblock Cambridge university press, 1998.

\bibitem{Bal}
R.~Balaji and R.B. Bapat.
\newblock On {E}uclidean distance matrices.
\newblock {\em Linear Algebra Appl.}, 424:108--117, 2007.

\bibitem{Bapat}
R.B. Bapat.
\newblock {\em Graphs and {M}atrices}.
\newblock Springer London, 2010.

\bibitem{Bas}
E.T. Baskoro and N.~Obata.
\newblock Determining finite connected graphs along the quadratic embedding constants of paths.
\newblock {\em Electronic Journal of Graph Theory and Applications}, 9(2):539--560, 2021.

\bibitem{ObataBaskoro2024}
E.T. Baskoro and N.~Obata.
\newblock A {C}lassification of {G}raphs {T}hrough {Q}uadratic {E}mbedding {C}onstants and {C}lique {G}raph {I}nsights.
\newblock {\em Communications in Combinatorics and Optimization}, 2024.

\bibitem{Boz2}
M.~Bożejko.
\newblock Positive-definite kernels, length functions on groups and a noncommutative von neumann inequality.
\newblock {\em Studia Mathematica}, 95:107--118, 1989.

\bibitem{BozDolEjsGal}
M.~Bożejko, M.~Doł\k{e}ga, W.~Ejsmont, and Ś. Gal.
\newblock Reflection length with two parameters in the asymptotic representation theory of type b/c and applications.
\newblock {\em Journal of Functional Analysis}, 284, 2022.

\bibitem{BozJanSpa}
M.~Bożejko, T.~Januszkiewicz, and R.J. Spatzier.
\newblock Infinite {C}oxeter groups do not have {K}azhdan's property.
\newblock {\em Journal of Operator Theory}, 19(1):63--67, 1988.

\bibitem{ChouNan}
P.~N. Choudhury and R~Nandi.
\newblock Quadratic {E}mbedding {C}onstants of graphs: Bounds and distance spectra.
\newblock {\em Linear Algebra and its Applications}, 680, 2023.

\bibitem{Deza1997}
M.~Deza, M.~Laurent, and R.~Weismantel.
\newblock {\em Geometry of cuts and metrics}, volume~2.
\newblock Springer, 1997.

\bibitem{Djokovic}
D.Ž Djoković.
\newblock Distance-preserving subgraphs of hypercubes.
\newblock {\em Journal of Combinatorial Theory, Series B}, 14:263–267, 1973.

\bibitem{Winkler}
R.~Graham and P.~Winkler.
\newblock On isometric embeddings of graphs.
\newblock {\em Proceedings of the National Academy of Sciences of the United States of America}, 81:7259--60, 1984.

\bibitem{Klavzar}
S.~Gravier, S.~Klavžar, and M.~Mollard.
\newblock Isometric embeddings of subdivided wheels in hypercubes.
\newblock {\em Discrete Mathematics}, 269:287--293, 2003.

\bibitem{Ham2011}
R.~Hammack, W.~Imrich, and S.~Klav{\v{z}}ar.
\newblock {\em Handbook of product graphs}, volume~2.
\newblock CRC press Boca Raton, 2011.

\bibitem{Hora1}
A.~Hora and N.~Obata.
\newblock {\em Quantum probability and spectral analysis of graphs. With aforeword by Professor Luigi Accardi}.
\newblock Springer, Berlin, 2007.

\bibitem{IraSug}
W.~Irawan and K.~Sugeng.
\newblock Quadratic embedding constants of hairy cycle graphs.
\newblock {\em Journal of Physics: Conference Series}, 1722:012046, 2021.

\bibitem{Jak1}
G.~Jaklič and J.~Modic.
\newblock On properties of cell matrices.
\newblock {\em Applied Mathematics and Computation}, 216:2016--2023, 2010.

\bibitem{Jak2}
G.~Jaklič and J.~Modic.
\newblock On {E}uclidean distance matrices of graphs.
\newblock {\em ELA. The Electronic Journal of Linear Algebra}, 26, 2013.

\bibitem{Klav2008}
S.~Klavzar.
\newblock A bird’s eye view of the cut method and a survey of its applications in chemical graph theory.
\newblock {\em MATCH Commun. Math. Comput. Chem}, 60(2):255--274, 2008.

\bibitem{Klav1997}
S.~Klav{\v{z}}ar and I.~Gutman.
\newblock Wiener number of vertex-weighted graphs and a chemical application.
\newblock {\em Discrete applied mathematics}, 80(1):73--81, 1997.

\bibitem{Lou}
Z.~Lou, N.~Obata, and Q.X. Huang.
\newblock Quadratic embedding constants of graph joins.
\newblock {\em Graphs and Combinatorics}, 38, 2022.

\bibitem{Maehara}
H.~Maehara.
\newblock Euclidean embeddings of finite metric spaces.
\newblock {\em Discrete Mathematics}, 313, 2013.

\bibitem{Mlo1}
W.~Młotkowski.
\newblock Quadratic embedding constants of path graphs.
\newblock {\em Linear Algebra and its Applications}, 644:95--107, 2022.

\bibitem{Mlo2}
W.~Młotkowski and N.~Obata.
\newblock On quadratic embedding constants of star product graphs.
\newblock {\em Hokkaido Mathematical Journal}, 49, 2018.

\bibitem{MłotkowskiObata2024}
W.~Młotkowski and N.~Obata.
\newblock Quadratic embedding constants of fan graphs and graph joins.
\newblock {\em arXiv:2403.16348v1}, 2024.

\bibitem{Obata5}
N.~Obata.
\newblock Positive {Q}-matrices of graphs.
\newblock {\em Studia Mathematica}, 179:81--97, 2007.

\bibitem{Obata6}
N.~Obata.
\newblock Markov product of positive definite kernels and applications to {Q}-matrices of graph products.
\newblock {\em Colloquium Mathematicum}, 122, 2011.

\bibitem{Obata4}
N.~Obata.
\newblock Quadratic embedding constants of wheel graphs.
\newblock {\em Interdisciplinary Information Sciences}, 23:171--174, 2017.

\bibitem{Obata1}
N.~Obata.
\newblock Complete multipartite graphs of non-{Q}{E} class.
\newblock {\em Electron. J. Graph Theory Appl.}, 11:511--527, 2022.

\bibitem{ObataPrim6v}
N.~Obata.
\newblock Primary non-{Q}{E} graphs on six vertices.
\newblock {\em Interdisciplinary Information Sciences}, 29, 2023.

\bibitem{Obata3}
N.~Obata and A.~Zakiyyah.
\newblock Distance matrices and quadratic embedding of graphs.
\newblock {\em Electronic Journal of Graph Theory and Applications}, 6:37, 2018.

\bibitem{Pur}
M.~Purwaningsih and K.~Sugeng.
\newblock Quadratic embedding constants of squid graph and kite graph.
\newblock {\em Journal of Physics: Conference Series}, 1722:012047, 2021.

\bibitem{RothWinkler}
R.~Roth and P.~Winkler.
\newblock Collapse of the metric hierarchy for bipartite graphs.
\newblock {\em Eur. J. Comb.}, 7:371--375, 1986.

\bibitem{Sch2}
I.J. Schoenberg.
\newblock Metric spaces and positive definite functions.
\newblock {\em Transactions of the American Mathematical Society}, 44:522--536, 1938.

\bibitem{Tan}
H.~Tanaka.
\newblock Characterizing graphs with fully positive semidefinite {Q}-matrices.
\newblock {\em Linear Algebra and its Applications}, 671, 2023.

\end{thebibliography}
\end{document}